\documentclass[10pt,reqno]{amsart}

\usepackage{amssymb}
\usepackage{amsthm}
\usepackage{amsmath}
\usepackage{mathrsfs}
\usepackage{comment}
\usepackage{graphicx}
\usepackage{float}

\usepackage[normalem]{ulem}

\usepackage{hyperref}

\numberwithin{equation}{section}
\allowdisplaybreaks

\theoremstyle{plain}
\newtheorem{proposition}{Proposition}[section]
\newtheorem{theorem}[proposition]{Theorem}
\newtheorem{lemma}[proposition]{Lemma}

\newtheorem{result}[proposition]{Result}

\theoremstyle{definition}
\newtheorem{example}[proposition]{Example}
\newtheorem{definition}[proposition]{Definition}

\theoremstyle{remark}
\newtheorem{remark}[proposition]{Remark}

\DeclareMathOperator{\Real}{Re}

\DeclareMathOperator{\id}{id}

\DeclareMathOperator{\Cc}{\mathcal{C}}

\DeclareMathOperator{\Oc}{\mathcal{O}}

\DeclareMathOperator{\Tc}{\mathcal{T}}
\DeclareMathOperator{\Uc}{\mathcal{U}}

\DeclareMathOperator{\Bb}{\mathbb{B}}
\DeclareMathOperator{\Cb}{\mathbb{C}}
\DeclareMathOperator{\Db}{\mathbb{D}}

\DeclareMathOperator{\Nb}{\mathbb{N}}

\DeclareMathOperator{\Rb}{\mathbb{R}}
\DeclareMathOperator{\Sb}{\mathbb{S}}

\DeclareMathOperator{\Zb}{\mathbb{Z}}

\DeclareMathOperator{\Ks}{{\sf K}}

\newcommand{\abs}[1]{\left|#1\right|}

\newcommand{\norm}[1]{\left\|#1\right\|}
\newcommand{\wt}[1]{\widetilde{#1}}

\newcommand{\Gp}[3]{\left({#1}\!\mid\!{#2}\right)^{\Omega}_{{#3}}} 
\newcommand{\grp}[3]{\left({#1}\!\mid\!{#2}\right)_{{#3}}}

\newcommand{\bcdot}{\boldsymbol{\cdot}}
\newcommand{\blz}{\blacklozenge} 
\newcommand{\nrg}[3]{{#1}\,|\,{#2}\!\frown\!{#3}} 
\newcommand{\glim}{\stackrel{G\,\,}{\rightarrow}}

\newcommand\clos[1]{{\overline{{#1}}}^{\!e}}

\begin{document}

\title[Unbounded visibility domains]{Unbounded visibility domains, \\ the end compactification, and applications}
\author{Gautam Bharali}
  \address{Department of Mathematics, Indian Institute of Science, Bangalore 560012, India}
  \email{bharali@iisc.ac.in}
\author{Andrew Zimmer}
  \address{Department of Mathematics, University of Wisconsin-Madison, Madison, WI 53706, U.S.A.}
  \email{amzimmer2@wisc.edu}
\date{\today}
\keywords{Kobayashi metric, visibility, Gromov hyperbolicity, end compactification, continuous extensions, quasi-isometries, Wolff--Denjoy theorem}
\subjclass[2020]{Primary 32F45, 53C23; Secondary 32Q45, 53C22, 32A40, 32H50}

\begin{abstract}
In this paper we study when the Kobayashi distance on a Kobayashi hyperbolic domain has certain visibility properties, with a focus on unbounded domains. ``Visibility'' in this context is reminiscent of visibility, seen in negatively curved Riemannian manifolds, in the sense of Eberlein--O'Neill. However, we do not assume that the domains studied are Cauchy-complete with respect to the Kobayashi distance, as this is hard to establish for domains in $\Cb^d$, $d \geq 2$. We study the various ways in which this property controls the boundary behavior of holomorphic maps. Among these results is a Carath{\'e}odory-type extension theorem for biholomorphisms between planar domains\,---\,notably: between infinitely-connected domains. We also explore connections between our visibility property and Gromov hyperbolicity of the Kobayashi distance. 
\end{abstract}

\maketitle

\section{Introduction}\label{sec:intro}

In this paper we study when {the Kobayashi distance on a domain has a certain visibility property, how this property controls the boundary behavior of holomorphic maps, and connections between the visibility property and Gromov hyperbolicity of the Kobayashi distance. Informally speaking, the visibility property states that geodesics joining two distinct points on the boundary must bend into the domain (as in the Poincar\'e disk model of the hyperbolic plane). This property has been used extensively, e.g. \cite{CHL1988,M1993,K2005b}, and has been systematically investigated in a number of recent papers, e.g.~\cite{BZ2017,BM2021,CMS2021,BNT2022}.

In earlier work \cite{BZ2017}, we had introduced a new class of domains, the \emph{Goldilocks domains}, proved that they have the visibility property, and then used this property to understand the behavior of holomorphic maps. In a part of this paper, we extend those ideas to unbounded domains. 

We now introduce the definitions needed to state our main results. In particular, we need to recall the definition of the Golidlocks condition and to precisely define the curves and the notion of boundaries used to define the visibility property.

Given a domain $\Omega \subset \Cb^d$, let $k_\Omega : \Omega \times \Cb^d \rightarrow [0,\infty)$ denote the \emph{infinitesimal Kobayashi pseudo-metric}, and let $K_\Omega: \Omega \times \Omega \rightarrow [0,\infty)$ denote the Kobayashi pseudo-distance. We say that $\Omega$ is \emph{Kobayashi hyperbolic} if $K_\Omega$ is an actual distance.

If the metric space $(\Omega, K_{\Omega})$ is Cauchy-complete (for brevity: $\Omega$ is \emph{complete Kobayashi hyperbolic}) then any two points in $\Omega$ are joined by a geodesic (i.e., a curve $\sigma : I \rightarrow \Omega$, where $I$ is an interval, that satisfies $K_\Omega(\sigma(t), \sigma(s)) = \abs{t-s}$ for all $s,t \in I$). However, when $d \geq 2$ it is a very difficult problem to determine if a given domain is complete Kobayashi hyperbolic (even in the pseudoconvex case). Therefore, the domains studied in this paper are \textbf{not} assumed to be complete Kobayashi hyperbolic. Hence, we need to consider a more general class of curves instead of geodesics.

Let $\Omega\subset \Cb^d$ be a domain. For $\lambda \geq 1$ and $\kappa \geq 0$, a map $\sigma : I \rightarrow \Omega$ of an interval $I \subset \Rb$ is called a \emph{$(\lambda, \kappa)$-almost-geodesic} if 
\begin{itemize}
\item $\lambda^{-1}\abs{t-s} - \kappa \leq K_\Omega(\sigma(t), \sigma(s)) \leq \lambda\abs{t-s} + \kappa$ for all $s,t \in I$, and
\item $\sigma$ is absolutely continuous (as a map $I \rightarrow \Cb^d$, whereby $\sigma^\prime(t)$ exists for almost every $t\in I$) and $k_\Omega(\sigma(t); \sigma^\prime(t)) \leq \lambda$ for almost every $t \in I$. 
\end{itemize}
These curves are relevant for the following reason: if $\Omega$ is Kobayashi hyperbolic, then for any $\kappa > 0$, every two points in $\Omega$ are joined by a $(1,\kappa)$-almost-geodesic\,---\,see Proposition~\ref{prop:almost_geod_exist} below. 

Given a domain $\Omega \subset \Cb^d$, let $\overline{\Omega}^{End}$ denote the end compactification of $\overline{\Omega}$. We shall write $\partial \overline{\Omega}^{End} = \overline{\Omega}^{End} \setminus \Omega$. The reader is referred to Section~\ref{sec:end_comp} for the definition of the end compactification. With these notions, we can now formally define the visibility property alluded to above.

\begin{definition}\label{defn:visib_dom}
Let $\Omega\subset \Cb^d$ be a domain (which is not necessarily bounded). We say that $\Omega$ is a \emph{visibility domain with respect to the Kobayashi distance} (or simply a \emph{visibility domain}) if $\Omega$ is Kobayashi hyperbolic and has the following property:
\begin{itemize}
\item[$(*)$] If $\lambda \geq 1$, $\kappa \geq 0$, $\xi, \eta \in \partial\overline{\Omega}^{End}$ are distinct points, and $V_\xi, V_\eta$ are $\overline{\Omega}^{End}$-open neighborhoods of $\xi, \eta$, respectively, whose closures in $\overline{\Omega}^{End}$ are disjoint, then there exists a compact set $K \subset \Omega$ such that for any $(\lambda, \kappa)$-almost-geodesic $\sigma: [0,T] \rightarrow \Omega$ with $\sigma(0) \in V_\xi$ and $\sigma(T) \in V_\eta$, $\sigma([0,T]) \cap K \neq \emptyset$.
\end{itemize}
For $\Omega$ as above, we say that $\Omega$ is a \emph{weak visibility domain} if the property $(*)$ holds true \textbf{only} for $\lambda=1$.
\end{definition}

Observe that a visibility domain in the above sense is a weak visibility domain. Before we turn to more concrete matters, it might be useful to address the following question: what is the significance of having two notions of visibility? Functionally, visibility may be seen as a tool for controlling the oscillation of a map $f: D \rightarrow \Omega$, if $f$ maps into a (weak) visibility domain, along any sequence $(z_n)_{n \geq 1} \subset D$ as $z_n$ approaches a point in $\overline{D}^{End}$. Provided $D$ has reasonably well-behaved geodesics (in a sense that can be made precise), this control facilitates the continuous extension of $f$ to a map between $\overline{D}^{End}$ and $\overline{\Omega}^{End}$. From this perspective, loosely speaking:
\begin{itemize}
\item weak visibility with respect to $K_{\Omega}$ is the property that enables continuous extension, as described above, for isometries with respect to $K_{D}$ and $K_{\Omega}$,
\item visibility with respect to $K_{\Omega}$ is the property that enables continuous extension, as described above, for continuous \textbf{quasi}-isometries with respect to $K_{D}$ and $K_{\Omega}$.
\end{itemize}
  
Gromov hyperbolicity is another framework in which the above extension phenomena are obtained: e.g., see \cite[Section~6]{BB2000}. It turns out that there is a natural relationship between weak visibility and Gromov hyperbolicity of the Kobayashi distance, which we shall investigate and present an application thereof. On the theme of the last paragraph: we shall establish a rather general result on the homeomorphic extension between $\overline{\Omega}_1^{End}$ and $\overline{\Omega}_2^{End}$ of biholomorphisms between planar domains $\Omega_1$ and $\Omega_2$, the Poincar{\'e} distances on which need not be Gromov hyperbolic. Its proof relies crucially on visibility in the sense of Definition~\ref{defn:visib_dom}. In fact, a portion of this paper is devoted to planar domains, Gromov hyperbolicity of the Poincar{\'e} distance (or the failure thereof) on these domains, etc.

\subsection{The local Goldilocks property}
Given the above discussion, it would be useful to have sufficient conditions for a domain to be a visibility domain. The Goldilocks property referred to above is sufficient for a bounded domain to be a visibility domain. We begin with notation needed to extend these ideas to unbounded domains. Let $\Omega \subset \Cb^d$ be a domain. Given a subset $U \subset \overline{\Omega}$, we define 
 
\begin{align*}
M_{\Omega,U}(r) := \sup\left\{ \frac{1}{k_{\Omega}(z;v)} : z \in \Omega \cap U, d_{{\rm Euc}}(z,\partial \Omega) \leq r, \norm{v}=1\right\}
\end{align*}
to measure the growth of the Kobayashi pseudo-metric as one approaches $\partial\Omega$ through $\Omega 
\cap U$. In \cite{BZ2017}, we used the asymptotic behavior of the Kobayashi distance and metric to define the following class of domains (in what follows, we will abbreviate $d_{\rm Euc}(z,\partial \Omega)$ as $\delta_{\Omega}(z)$). 

\begin{definition} A bounded domain $\Omega \subset \Cb^d$ is a \emph{Goldilocks domain} if 
\begin{enumerate}
\item for some (hence any) $\epsilon >0$ we have
\begin{align*}
\int_0^\epsilon \frac{1}{r} M_{\Omega,\Omega}\left(r\right) dr < \infty,
\end{align*} 
\item for each $z_0 \in \Omega$ there exist constants $C, \alpha > 0$ such that 
\begin{align*}
K_\Omega(z_0, z) \leq C + \alpha \log \frac{1}{\delta_\Omega(z)}
\end{align*}
for all $z \in \Omega$. 
 \end{enumerate}
 \end{definition} 
 
The second condition can be viewed as a type of regularity condition on the boundary and holds, for instance, if $\partial \Omega$ is $\Cc^{0,1}$-smooth }(which can be inferred from \cite[Lemma~2.3]{BZ2017}). The first condition can be viewed as a uniform obstruction to analytic varieties in the boundary: for instance if $\partial \Omega$ is reasonably regular and there exists a non-constant holomorphic map $\Db \rightarrow \partial \Omega$, then one can show that  $\liminf_{r \searrow 0} M_{\Omega, \Omega}(r) > 0$ and hence the first condition fails. Here, for $\partial\Omega$ to be ``reasonably regular'', it sufficies for $\Omega$ to be a $\Cc^{0}$-domain (see Definition~\ref{def:Lip_domain}) and the last observation follows, essentially, from the argument at the beginning of the proof of Theorem~\ref{thm:Gromov_visi_bdy} and from the estimate \eqref{eqn:deriv_1}.  

To study the visibility property in unbounded domains, we shall use the following localized version of the Goldilocks conditions. This idea was introduced in \cite{CMS2021}, although the domains considered in \cite{CMS2021} in this context are still bounded domains. When considering unbounded domains, certain fundamental difficulties arise, which must be managed when proving the main result of this section (also see Remark~\ref{rem:CMS}).
 
\begin{definition}\label{defn:local_GL}
Let $\Omega \subset \Cb^d$ be a domain. A boundary point $x \in \partial \Omega$ is a \emph{local Golidlocks point} if there exists a neighborhood $U$ of $x$ in $\overline{\Omega}$ such that
\begin{enumerate}
\item\label{item:metric} for some (hence any) $\epsilon >0$ we have
\begin{align*}
\int_0^\epsilon \frac{1}{r} M_{\Omega,U}\left(r\right) dr < \infty,
\end{align*} 
\item\label{item:dist} for each $z_0 \in \Omega$ there exist constants $C, \alpha > 0$ (which depend on $z_0$ and $U$) such that 
\begin{align*}
K_\Omega(z_0, z) \leq C + \alpha \log \frac{1}{\delta_\Omega(z)}
\end{align*}
for all $z \in \Omega \cap U$. 
\end{enumerate}
Let $\partial_{{ \rm lg}} \Omega \subset \partial \Omega$ denote the set of local Golidlocks points. We say that \emph{$\Omega$ is locally a Goldilocks domain} if $\partial_{{ \rm lg}} \Omega = \partial \Omega$
\end{definition}

With these definitions, we are ready to state the main result of this section. 

\begin{theorem}\label{thm:visible}
Let $\Omega \subset \Cb^d$ be a Kobayashi hyperbolic domain. Suppose the set $\partial\Omega \setminus \partial_{{\rm lg}}\Omega$ is totally disconnected. Then, $\Omega$ is a visibility domain with respect to the Kobayashi distance.
\end{theorem}

\begin{remark}\label{rem:CMS}
Theorem~\ref{thm:visible} is similar in spirit to \cite[Theorem 1.9]{CMS2021}, which states that if a local property similar to that in Definition~\ref{defn:local_GL} holds around points outside a sufficiently small set $S\varsubsetneq \partial\Omega$, then $\Omega$ is a visibility domain. That said:
\begin{enumerate}
\item Only bounded domains are considered in \cite[Theorem 1.9]{CMS2021}. Moreover, the above-mentioned $S$ is such that the case when $\partial\Omega \setminus \partial_{{\rm lg}}\Omega$ is, say, a Cantor set in $\partial\Omega$ is not covered by \cite[Theorem 1.9]{CMS2021}.
\item The ``local property similar to$\dots$'' that was alluded to is a localized version of a condition introduced in \cite{BM2021}. The conclusion of Theorem~\ref{thm:visible} can be deduced with the latter property replacing our condition determining the set $\partial_{{\rm lg}}\Omega$. However, since new challenges arise in proving Theorem~\ref{thm:visible} when $\Omega$ is unbounded, we have opted for a hypothesis wherein the ideas used in the proof are the clearest rather than for the most general statement.
\end{enumerate}
\end{remark}   

Given a domain $\Omega\subset \Cb^d$, \cite[Section~2]{BZ2017} presents a variety of geometric conditions that $\partial\Omega$ can satisfy (locally) around a point $x\in \partial\Omega$ for $x$ to be a local Goldilocks point. Here, we present a range of new examples of locally Goldilocks domains in Section~\ref{sec:examples}.

\subsection{Applications to (quasi-)isometric maps}\label{ssec:applications_maps}
The next couple of results substantiate the discussion above on the functional significance of the visibility property. So, these results pertain to continuous\,---\,or even better\,---\,extension of different types of maps into a visibility domain. Deferring the definition of ``well-behaved geodesics'' to Section~\ref{sec:quasi_isom_extn} below, we state the first of our extension theorems. 

\begin{theorem}\label{thm:extensions}
Suppose $\Omega_1 \subset \Cb^{d_1}$, $\Omega_2 \subset \Cb^{d_2}$ are domains where 
\begin{enumerate}
\item $\Omega_1$ has well-behaved geodesics, 
\item $\Omega_2$ is a visibility domain.
\end{enumerate}
If $f : \Omega_1 \rightarrow \Omega_2$ is a continuous quasi-isometric embedding relative to the Kobayashi distances on $\Omega_1$ and $\Omega_2$, then $f$ extends to a continuous map $\overline{\Omega}^{End}_1 \rightarrow \overline{\Omega}^{End}_2$. 
\end{theorem}

\begin{remark}
We should point out here that in the above theorem, $\Omega_1$ is complete Kobayashi hyperbolic. This is a part of the condition that $\Omega_1$ has well-behaved geodesics.
\end{remark}

To state our next result, we need a definition.

\begin{definition}\label{def:Lip_domain}
Let $\Omega \subset \Cb^d$ be a domain. We say that $\Omega$ is a \emph{Lipschitz domain} (resp., a \emph{$\Cc^{0}$ domain}) if for every $x \in \partial \Omega$, there is a neighborhood $U_x$ of $x$, a unitary change of coordinates centered at $x$, and a Lipschitz function (resp., continuous function) $\varphi_x$ on some open neighborhood of $0$ in $\Cb^{d-1} \times \Rb$ such that, if $w=(w_1,\dots, w_n)$ denotes these centered coordinates and $W_x\ni 0$ the range of this chart, then $U_x \cap \Omega$ relative to these coordinates is given by $\{(w_1,\dots, w_n)\in W_x : {\sf Im}(w_n) > \varphi_x(w_1,\dots, w_{n-1}, {\sf Re}(w_n))\}$.
\end{definition}

\begin{remark}\label{rem:about_C0_Lip_dom}
A domain $\Omega\varsubsetneq \Cb^d$ such that $\partial\Omega$ is an embedded Lipschitz submanifold of $\Rb^{2d}$ is not necessarily a Lipschitz domain. However, the latter terminology is standard (with the unitary changes of coordinates mentioned in Definition~\ref{def:Lip_domain} replaced by orthogonal changes of coordinates in the case of domains in $\Rb^N$, $N\geq 2$). Such domains have many pleasant properties and have been studied extensively (see, e.g., \cite{A1975} and the references therein). Similarly, a domain $\Omega\varsubsetneq \Cb^d$ such that $\partial\Omega$ is an embedded topological submanifold of $\Rb^{2d}$ is not necessarily a $\Cc^{0}$ domain. The latter statement and the first sentence of this remark are a part of \cite[Theorem~1.2.1.5]{G1985} (with the understanding that our domains are subsets of $\Cb^d$ and that the local changes of coordinates mentioned in Definition~\ref{def:Lip_domain} replace the orthogonal changes of coordinates that are a part of the definitions in \cite{G1985}). An example of a domain $\Omega \varsubsetneq \Cb$ whose boundary is an embedded Lipschitz submanifold and such that $\Omega$ is not even a $\Cc^{0}$ domain (and, hence, not a Lipschitz domain) is given in \cite[pp.~7--9]{G1985}.
\end{remark} 

We are now able to state our next extension theorem. It is a part of the focus of this paper, alluded to above, on planar domains.

\begin{theorem}\label{thm:planar_domains}
Let $\Omega_1, \Omega_2 \varsubsetneq \Cb$ be Lipschitz domains. If $f : \Omega_1 \rightarrow \Omega_2$ is a biholomorphism, then $f$ extends to a homeomorphism $\overline{\Omega}_1^{End} \rightarrow \overline{\Omega}_2^{End}$. 
\end{theorem}

This result is similar in spirit to Carath{\'e}odory's extension theorem for Riemann mappings. The simply connected domains to which the latter is applicable can have less regular boundaries than those of the domains in Theorem~\ref{thm:planar_domains}, but observe that $\Omega_1$ and $\Omega_2$ \textbf{need not} be simply connected. In fact, note that the domains for which Theorem~\ref{thm:planar_domains} holds true need not even be finitely-connected. Considerations very different from those in the proof of Carath{\'e}odory's extension theorem feature in the proof of Theorem~\ref{thm:planar_domains}: its proof relies crucially on visibility in the sense of Definition~\ref{defn:visib_dom}.

Gromov hyperbolicity is, in principle, a framework for establishing results like Theorem~\ref{thm:planar_domains}. However, it is rather easy to construct planar domains where  the Poincar{\'e} distance (which equals the Kobayashi distance) is not Gromov hyperbolic. In particular, in Section~\ref{sec:non-G-h} we give examples of planar domains that have $\Cc^{\infty}$ boundary on which the Poincar{\'e} distance is not Gromov hyperbolic. These examples are locally Goldilocks domains; by Theorem~\ref{thm:visible}, therefore, they are visibility domains with respect to the Kobayashi distance.

\subsection{Gromov hyperbolic spaces}
We now turn to the relationship between Gromov hyperbolicity and visibility alluded to earlier. We assume that the reader has some familiarity with Gromov hyperbolic metric spaces. For the present discussion, the metric space of interest is $(\Omega, K_{\Omega})$, where $\Omega$ is Kobayashi hyperbolic and\,---\,we reiterate\,---\,not necessarily bounded. Since we do not assume that $\Omega$ is complete Kobayashi hyperbolic, a few words are in order. In what follows, $\Gp{z}{w}{o}$ will denote the Gromov product, with respect to a base-point $o\in \Omega$, on $(\Omega, K_{\Omega})$ (see Definition~\ref{defn:G_seq}). Then, $(\Omega, K_{\Omega})$ is said to be \emph{Gromov hyperbolic} if there exists a $\delta\geq 0$ such that, for any four points $a, b, c, o \in \Omega$,
\begin{equation}\label{eq:Gromov_ineq}
\min\left\{ \Gp{a}{b}{o}, \Gp{b}{c}{o} \right\} \leq \Gp{a}{c}{o} + \delta.
\end{equation}
Loosely speaking, the above condition encodes the idea that, metrically, $(\Omega, K_{\Omega})$ ``approximately resembles'' an $\Rb$-tree.

To state the first result of this section, we need to introduce the concept of the Gromov boundary. Every Gromov hyperbolic space $(X, d)$ has an abstract boundary, called the \emph{Gromov boundary} and denoted by $\partial_{G}X$, and a topology on $(X \cup \partial_{G}X)$ that compactifies $X$ if $(X, d)$ is a proper geodesic space. For a full description of the set $\partial_{G}X$, and a discussion of the topology on $(X \cup \partial_{G}X)$, we refer the reader to Section~\ref{ssec:Gromov_top}. Having stated these preliminaries, we present the following result:

\begin{theorem}\label{thm:Gromov_visi}
Let $\Omega \subset \Cb^d$ be a Kobayashi hyperbolic domain and suppose $(\Omega, K_{\Omega})$ is Gromov hyperbolic. If the identity map ${\sf id}_{\Omega}$ extends to a homeomorphism from $(\Omega \cup \partial_{G}\Omega)$ onto $\overline{\Omega}^{End}$, then:
\begin{enumerate}
\item\label{item:complete} $(\Omega, K_{\Omega})$ is Cauchy-complete,
\item\label{item:wvd} $\Omega$ is a weak visibility domain.
\end{enumerate}
\end{theorem}

Theorem~\ref{thm:Gromov_visi} is reminiscent of one part of a result by Bracci \emph{et al.} \cite[Theorem~3.3]{BNT2022} (also see \cite[Theorem~1.4]{CMS2021} by Chandel \emph{et al.}) which states that for a bounded, complete Kobayashi hyperbolic (hence geodesic) domain $\Omega \subset \Cb^d$ such that $(\Omega, K_{\Omega})$ is Gromov hyperbolic, if ${\sf id}_{\Omega}$ extends to a homeomorphism from $(\Omega \cup \partial_{G}\Omega)$ onto $\overline{\Omega}$, then $\partial\Omega$ possesses a type of visibility (analogous to what $\partial_{G}\Omega$ possesses; see \cite[Lemma~III.H-3.2]{BH1999} for details). It turns out that with the assumptions just stated, one can deduce a form of weak visibility wherein the condition $(*)$ in Definition~\ref{defn:visib_dom} holds true just for geodesics \cite{BNT2022*}. In contrast to \cite{BNT2022, CMS2021}:
\begin{itemize}
\item We do not, in Theorem~\ref{thm:Gromov_visi}, assume \emph{a priori} that $(\Omega, K_{\Omega})$ is Cauchy-complete.
\item For the latter reason\,---\,despite the interesting conclusion (\ref{item:complete})\,---\,the natural visibility property one would like to deduce for $\Omega$ is weak visibility (as given by Definition~\ref{defn:visib_dom}). This entails the harder task of establishing the control described by Definition~\ref{defn:visib_dom} for the relevant $(1, \kappa)$-almost-geodesics for \textbf{all $\boldsymbol{\kappa \geq 0}$.}
\end{itemize}
Establishing this control for $(1, \kappa)$-almost-geodesics for all $\kappa \geq 0$ isn't a mere curiosity. Weak visibility of $\Omega$ turns out to be crucial in proving the next result (refer to the observation following Theorem~\ref{thm:Gromov_visi_bdy} below).
 
With the hypothesis of Theorem~\ref{thm:Gromov_visi}, a uniform slim-triangles condition on $(\Omega, K_{\Omega})$, for triangles whose sides are $(1, \kappa)$-almost-geodesics, would provide the most intuitive argument for $\Omega$ being a a weak visibility domain. To this end, for $\Omega \subset \Cb^d$ Kobayashi hyperbolic and such that $(\Omega, K_{\Omega})$ is Gromov hyperbolic, we establish a slim-triangles formulation\,---\,involving $(1, \kappa)$-almost-geodesics\,---\,for the Gromov hyperbolicity of $(\Omega, K_{\Omega})$. We refer the reader to Section~\ref{ssec:slim} for an exact theorem that establishes this, and which might also be of independent interest.

The next result is motivated by the speculation by Balogh--Bonk \cite[Section~6]{BB2000} that Gromov hyperbolicity of $(\Omega, K_{\Omega})$ may be provable for $\Omega$ a smoothly-bounded pseudoconvex domain of finite type. In \cite{BB2000}, they show that when $\Omega$ is a bounded strongly pseudoconvex domain with $\Cc^2$-smooth boundary, then $(\Omega, K_{\Omega})$ is Gromov hyperbolic. The second author proved \cite{Z2014} that if $\Omega$ is a bounded convex domain with $\Cc^\infty$-smooth boundary, then $(\Omega, K_{\Omega})$ is Gromov hyperbolic if and only if $\Omega$ is of finite type. Recently, Fiacchi \cite{F2022} showed that if $\Omega\subset \Cb^2$ is a bounded pseudoconvex domain with $\Cc^\infty$-smooth boundary and is of finite type, then $(\Omega, K_{\Omega})$ is Gromov hyperbolic. Attempts to extend the last result to higher dimensions encounter considerable technical complications. Furthermore, very little is known, beyond the main theorem in \cite{BB2000}, about the Gromov hyperbolicity of $(\Omega, K_{\Omega})$ when $\partial\Omega$ has low regularity (however,
see \cite[Theorem~1.6]{Z2017}, and see \cite{NTT2016, NT2018, PZ2018} for instances of non-Gromov-hyperbolic domains). It would thus be of interest to identify obstacles to the Gromov hyperbolicity of the Kobayashi distance. Theorem~\ref{thm:Gromov_visi_bdy} identifies such an obstacle. In what follows, we will refer to a complex subvariety of some (typically small) open set intersecting $\partial\Omega$ as a \emph{germ of a complex variety}.

\begin{theorem}\label{thm:Gromov_visi_bdy}
Let $\Omega \subset \Cb^d$ be a Kobayashi hyperbolic $\Cc^{0}$ domain.
If $(\Omega, K_{\Omega})$ is Gromov hyperbolic and the identity map ${\sf id}_{\Omega}$ extends to a homeomorphism from $(\Omega \cup \partial_{G}\Omega)$ onto $\overline{\Omega}^{End}$, then $\partial\Omega$ does not contain any germs of complex varieties of positive dimension.
\end{theorem}

In view of Theorem~\ref{thm:Gromov_visi}, the domain $\Omega$ in the above theorem is a weak visibility domain. Its proof now follows from the fact that if $\partial\Omega$ contained a germ of a complex variety, then there would exist a number $T>0$ and a sequence of $(1,T)$-almost-geodesics $(\sigma_n)_{n\geq 1}$ whose behavior violates the condition of weak visibility.

We should point out that one cannot, in general, omit the condition on the extension of the map ${\sf id}_{\Omega}$ in Theorem~\ref{thm:Gromov_visi_bdy}. To see this, we refer to \cite[Proposition~1.9]{Z2017} which presents, for each $d \geq 2$, an example of a $\Cb$-convex domain $\Omega \subset \Cb^d$ such that $(\Omega, K_{\Omega})$ is Gromov hyperbolic and $\partial\Omega$ contains a complex affine ball of dimension $(d-1)$. However, more can be said in the case of convex domains. Gaussier--Seshadri \cite{GS2018} showed that the presence of a non-trivial holomorphic disk in the boundary of a $\Cc^\infty$-smoothly bounded convex domain $\Omega \subset \Cb^d$ is an obstruction to $(\Omega, K_{\Omega})$ being Gromov hyperbolic. This result was extended by the second author to all convex domains that are Kobayashi hyperbolic without any regularity assumptions on their boundaries \cite[Theorem~1.6]{Z2014}.

\subsection{Wolff--Denjoy theorems}
There has long been interest in understanding the behavior of iterates of holomorphic self-maps of domains in complex Euclidean space. Unlike the chaotic behavior often seen in complex dynamics, iterates of holomorphic self-maps of Kobayashi hyperbolic domains often have very simple behavior. This is best demonstrated by the classical theorem of Wolff--Denjoy for holomorphic self-maps of the unit disk. 

\begin{result}[\cite{D1926, W1926}]\label{res:WD}
Suppose $f:\Db \rightarrow \Db$ is a holomorphic map then either:
\begin{enumerate}
\item $f$ has a fixed point in $\Db$; or
\item there exists a point $\xi \in \partial \Db$ such that
\begin{equation*}
\lim_{n \rightarrow \infty} f^n(z) = \xi
\end{equation*}
for any $z \in \Db$, this convergence being uniform on compact subsets of $\Db$.
\end{enumerate}
\end{result}

The above result was extended to the unit (Euclidean) ball in $\Cb^d$, for all $d$, by Herv{\'e} \cite{H1963}.   
It was further generalized by Abate\,---\,see \cite{A1988} or \cite[Chapter~4]{A1989}\,---\,to bounded strongly convex domains. The latter result was further generalized to a variety of bounded convex domains with progressively weaker assumptions  (see \cite{A2014} and the references therein), and even to certain holomorphic self-maps of open unit balls of reflexive Banach spaces that are reasonably ``nice'' (see for instance~\cite{BKR2013} and the references therein). The main result in \cite{H1984} by Huang is one of the few results that generalize Result~\ref{res:WD} to a class of domains in $\Cb^d$ that includes domains that are non-convex: namely, to (topologically) contractible strongly pseudoconvex domains. To put in perspective the hypothesis of Huang's result: in \cite{AH1992}, Abate--Heinzner  constructed a bounded contractible pseudoconvex domain that is strongly pseudoconvex except at one boundary point and admits a periodic automorphism having no fixed points. Wolff--Denjoy-type theorems are also known to hold on certain metric spaces where a boundary at infinity replaces the topological boundary; see for instance~\cite{K2001} or~\cite{B1997}. 

It is not hard to see that the dichotomy in Result~\ref{res:WD} fails in general if the domain considered is not contractible. An appropriate dichotomy that is suited to more general situations was  introduced by Abate in \cite{A1991} and, more recently, is seen in Wolff--Denjoy-type theorems established in \cite{BZ2017} and \cite{BM2021}. But, to the best of our knowledge, there are no versions of the above results for unbounded domains in the literature. Using the framework of this paper, we will demonstrate the following Wolff--Denjoy-type theorem:

\begin{theorem}\label{thm:WD_visible} 
Suppose $\Omega \subset \Cb^d$ is a taut domain. If $\Omega$ is a weak visibility domain and $F : \Omega \rightarrow \Omega$ is a holomorphic self-map, then either 
\begin{enumerate}
\item for any $z \in \Omega$, the orbit $\{ F^n(z) : n \geq 1\}$ is relatively compact in $\Omega$; or 
\item there exists $\xi \in \partial\overline{\Omega}^{End}$ such that 
$$
\lim_{n \rightarrow \infty} F^n(z) = \xi
$$
for all $z \in \Omega$, this convergence being uniform on compact subsets of $\Omega$.
\end{enumerate}
\end{theorem} 

We emphasize that Theorem~\ref{thm:WD_visible} does not assume that $\Omega$ is a complete Kobayashi hyperbolic domain; it merely requires $\Omega$ to be taut. For bounded domains, Theorem~\ref{thm:WD_visible} was established in~\cite{BM2021} using ideas from~\cite{BZ2017}. In contrast, Theorem~\ref{thm:WD_visible} is applicable to unbounded domains that satisfy the hypothesis of Theorem~\ref{thm:visible}\,---\,and thus to domains with very wild boundaries described in Section~\ref{sec:examples}.

\subsection{Frequently used notation}
The following notation will recur throughout this paper.
\begin{enumerate}
\item For $v \in \Cb^d$, $\norm{v}$ will denote the Euclidean norm of $v$. (Also, the phrase \emph{unit vector} will refer to any vector $v \in \Cb^d$ with $\norm{v}=1$.)
\item $\Bb_d(z,r)$ will denote the open Euclidean ball in $\Cb^d$ with center $z$ and radius $r$. However, for simplicity of notation:
  \begin{itemize}
    \item we shall write $\Bb_d(z,r)$ as $\Bb(z,r)$ if $d=1$,
    \item we shall write $\Bb_d(0,1)$ as simply $\Bb_d$ when $d\geq 2$, and
    \item we shall denote the open unit disk in $\Cb$ with center $0$ (i.e., $\Bb(0,1)$) as $\Db$.
  \end{itemize}
\item If $\Omega$ is a domain in $\Cb^d$, $z \in \Omega$, and $\sigma: [a, b] \rightarrow \Omega$ is a curve, we shall abbreviate $K_{\Omega}\big(z, \sigma([a, b])\big)$\,---\,i.e., the Kobayashi distance between $z$ and the image of $\sigma$\,---\,as $K_{\Omega}(z, \sigma)$.
\end{enumerate}

\section{Examples}\label{sec:examples}
This section is dedicated to various examples of domains that are ``irregular'' in specific ways but the geometry of whose boundaries satisfy the condition in Theorem~\ref{thm:visible}. Statements in which the latter is shown for some broad class of domains will be labeled as propositions, while specific constructions will be labeled as examples.

\subsection{The upper bound on the Kobayashi distance} 
The second condition in Definition~\ref{defn:local_GL} is quite mild and is satisfied for domains whose boundaries satisfy mild regularity conditions. 
 
An \emph{open right circular cone with aperture $\theta$} is an open subset of $\Cb^d$ of the form
\begin{align*}
\Gamma(v, \theta):= \{z\in \Cb^d : \Real[\,\langle z, v\rangle\,] > \cos(\theta/2)\norm{z}\},
\end{align*}
where $v$ is some unit vector in $\Cb^d$, $\theta\in (0, \pi)$, and $\langle\bcdot\,,\,\bcdot\rangle$ is the standard
Hermitian inner product on $\Cb^d$. 

\begin{definition}\label{def:cone_cond}
Let $\Omega$ be a domain in $\Cb^d$. We say that $\Omega$ satisfies a \emph{local interior-cone condition} if for each $R > 0$ there exist constants  $r_0 > 0$, $\theta\in (0, \pi)$, and a compact subset $K\subset \Omega$ (which depend on $R$) such that for each $z\in (\Omega\setminus K) \cap \Bb_d(0,R)$, there exist a point $\xi_z\in \partial\Omega$ and a unit vector $v_z$ such that
\begin{itemize}
 \item $z = \xi_z+t \bcdot v_z$ for some $t \in (0,r_0)$, and
 \item $(\xi_z+\Gamma(v_z, \theta))\cap \Bb_d(\xi_z, r_0) \subset \Omega$.
\end{itemize}
\end{definition}
 
Using the proof of \cite[Lemma 2.3]{BZ2017} or \cite[Proposition~2.5]{FR1987} or \cite[Proposition~2.3]{M1993}, one can establish the following: 

\begin{lemma}\label{lem:int_cone}
Let $\Omega$ be a  domain in $\Cb^d$ that satisfies a local interior-cone condition. Then, for each $x\in \partial\Omega$ and any specified $\overline{\Omega}$-open neighborhood $U$ of $x$ such that $U\Subset \Cb^d$, Condition~\ref{item:dist} in the definition of a local Goldilocks point is satisfied.
\end{lemma}

The only adaptation of the proof of \cite[Lemma 2.3]{BZ2017} that is required to prove the above is\,---\, having fixed a point $x\in \partial\Omega$ and an $\overline\Omega$-open neighborhood $U$ of $x$\,---\,to choose $R > 0$ so large that $U\Subset \Bb_d(0,R)$, following which the latter proof can be followed verbatim using the $\theta > 0$ and the compact $K$ determined by this $R$.

\subsection{Planar domains}\label{ssec:planar}
The focus of this section is a pair of constructions of domains with rough boundaries.

\begin{definition}\label{def:ext_cone_cond}
Let $\Omega$ be a domain in $\Cb$. We say that $\Omega$ satisfies a \emph{local exterior-cone condition} if for each $R > 0$ there exist constants  $r_0 > 0$, $\theta\in (0, \pi)$ such that for each $x \in \partial \Omega \cap \Bb(0,R)$, there exists a unit vector $v_x$ such that
$$
(x+\Gamma(v_x, \theta))\cap \Bb(x, r_0) \subset \Cb \setminus \Omega.
$$
\end{definition}

\begin{remark}\label{rem:about_Lip_dom}
This is a continuation of Remark~\ref{rem:about_C0_Lip_dom}. Concerning the pleasant properties mentioned in Remark~\ref{rem:about_C0_Lip_dom} that Lipschitz domains possess: elementary arguments show that Lipschitz domains satisfy a local interior-cone condition (in the sense of Definition~\ref{def:cone_cond}) and a local exterior-cone condition.
\end{remark}  

\begin{lemma}\label{lem:ext_cone}
Let $\Omega  \varsubsetneq \Cb$ be a domain that satisfies a local exterior-cone condition. Then, for each $R > 0$ there exists $c > 0$ such that if $z \in \Omega \cap \Bb(0,R)$ and $v \in \Cb$, then 
\begin{align*}
k_\Omega(z;v) \geq \frac{c\abs{v}}{\delta_\Omega(z)}.
\end{align*}
\end{lemma} 
\begin{proof}
First consider the case when $\Bb(0,2R) \subset \Omega$. Since $\#(\Cb \setminus \Omega) \geq 2$, $k_{\Omega}$ is given by a Riemannian metric (this is because, as $\Omega$ is a planar hyperbolic domain, $k_{\Omega}$ coincides with the Poincar\'e metric of $\Omega$). So, as $\overline{\Bb(0,R)} \subset \Omega$ is a compact subset, there exists, by continuity, some $c_0 > 0$ such that 
\begin{align*}
k_\Omega(z;v) \geq c_0\abs{v}
\end{align*}
for all $z \in \overline{\Bb(0,R)}$ and $v \in \Cb$. Hence, if $z \in \Omega \cap \Bb(0,R)$ and $v \in \Cb$, then 
\begin{align*}
k_\Omega(z;v) \geq c_0\abs{v} \geq c_0 R \frac{\abs{v}}{\delta_\Omega(z)}.
\end{align*}
 
Suppose that $\Bb(0,2R) \not\subset \Omega$. Fix $z \in \Omega \cap \Bb(0,R)$. Let $x \in \partial \Omega$ be a point with $\abs{z-x} = \delta_\Omega(z)$. By assumption $\abs{x} \leq 3R$. By the  exterior-cone condition there exist $\theta \in (0,\pi)$, $r_0 > 0$, and a unit vector $u$ such that 
$$(x+\Gamma(u, \theta))\cap \Bb(x, r_0) \subset \Cb \setminus \Omega.$$
Moreover, we can choose $\theta$ and $r_0$ to only depend on $R$. 

Let $t := \min\{ \delta_\Omega(z), r_0/2\}$ and $w := x+ tu$. Then there exists $c_1 \in (0,2)$ which only depends on $r_0,\theta$ such that 
\begin{equation}\label{eq:cone_rel}
\delta_{\Cb \setminus \Omega}(w) \geq c_1 \delta_\Omega(z).
\end{equation}
By this choice of $c_1$,
$$
c_1 \delta_\Omega(z) < \abs{z-w} \leq 2 \delta_\Omega(z),
$$
and $\Bb(w,c_1\delta_\Omega(z)) \subset \Cb \setminus \Omega$. 

Next, consider the holomorphic embedding $f$ given by 
$$
f(\zeta) = \frac{c_1\delta_\Omega(z)}{\zeta-w}, \quad \zeta \in \Omega.
$$
In view of \eqref{eq:cone_rel}, $f : \Omega \rightarrow \Db$. Then 
\begin{align*}
k_\Omega(z;v) \geq k_{\Db}(f(z);f^\prime(z)v) = \frac{\abs{f^\prime(z)}\abs{v}}{1-\abs{f(z)}^2}
&\geq \frac{1}{2}\,\frac{\abs{f^\prime(z)} \abs{v}}{1-\abs{f(z)}}   \\
&= \frac{1}{2}\,\frac{1}{\abs{z\!-\!w}-c_1\delta_\Omega(z)}\,\frac{c_1 \delta_\Omega(z)}{\abs{z\!-\!w}} \abs{v} \\
&\geq   \frac{1}{2}\,\frac{1}{2\delta_\Omega(z)\!-\!c_1\delta_\Omega(z)}\,\frac{c_1 \delta_\Omega(z)}
{2\delta_\Omega(z)} \abs{v}\\
&= \frac{c_1}{4(2\!-\!c_1)\delta_\Omega(z)}\abs{v}. 
\end{align*}
Since $c_1\in (0,2)$ does not depend on $z \in \Omega \cap \Bb(0,R)$, this completes the proof. 
\end{proof}

\begin{proposition}\label{prop:planar}
Let $\Omega \varsubsetneq \Cb$ be a Lipschitz domain
Then, every boundary point of $\Omega$ is a local Goldilocks point. 
\end{proposition}
\begin{proof}
Fix $x\in \partial\Omega$. Pick an $R > 0$ sufficiently large that $x \in \Bb(0, R)$ and \textbf{fix} it. From Remark~\ref{rem:about_Lip_dom} and from Lemma~\ref{lem:ext_cone}, it follows that if we write
$$
U := \overline{\Omega}\cap \Bb(0; R)
$$
then $U$ is an $\overline{\Omega}$-open neighborhood of $x$ and there exists a constant $c > 0$ such that
$$
k_\Omega(z;v) \geq \frac{c\abs{v}}{\delta_\Omega(z)}
$$
for any $z\in \Omega \cap U$ and for any $v\in \Cb$. This means that $M_{\Omega, U}(r) \leq cr$. Thus, Condition~\ref{item:metric} for $x$ to be a local Goldilocks point is satisfied. As $\Omega$ also satisfies a local interior-cone condition, by Lemma~\ref{lem:int_cone}, Condition~\ref{item:dist} for $x$ to be a local Goldilocks point is satisfied with the above choice of $U\ni x$. Thus, every boundary point of $\Omega$ is a local Goldilocks point.   
\end{proof}

Using Proposition~\ref{prop:planar}, it is easy to construct domains $\Omega \subsetneq \Cb$ with uncountably many ends for which every boundary point of $\partial\Omega$ is a local Goldilocks point. For instance, let $\Tc \subset \Cb$ be a proper embedding of an infinite tree with uncountably many ends; then we can construct a sufficiently small neighborhood, say $\Omega_{\Tc}$, of $\Tc$ such that $\partial \Omega_{\Tc}$ is a locally finite union of smooth curves. By Proposition~\ref{prop:planar}, every point in $\partial\Omega_{\Tc}$ is a local Goldilocks point.

\subsection{Domains in higher dimension with uncountably many ends}
Let us further develop this last construction, making use of its features to construct domains in $\Cb^2$.  

\begin{example} There exists a domain $\Omega \subset \Cb^2$ that is locally a Golidlocks domain such that $\overline{\Omega}$ has uncountably many ends.
\end{example}

\noindent{To construct such domains, we rely on the construction described at the end of Section~\ref{ssec:planar} to obtain a domain $\Omega_0 \varsubsetneq \Cb$ such that 
\begin{enumerate}
\item $0 \notin \overline{\Omega}_0$, 
\item $\Omega_0$ is a simply connected domain,
\item $\Omega_0$ is locally a Goldilocks domain,
\item $\overline{\Omega}_0$ has uncountably many ends, and
\item $\partial \Omega_0$ is a locally finite union of smooth curves. 
\end{enumerate}
Fix a biholomorphism $\varphi: \Db \rightarrow \Omega_0$. By Carath\'eodory's prime end theory\,---\,see Chapter~4 in~\cite{BCD2020}, for instance\,---\,$\varphi$ extends to a continuous homeomorphism $\overline{\Db} \rightarrow \overline{\Omega}_0^{End}$, which we shall also call $\varphi$. Let 
\begin{align*}
A := \varphi^{-1}\left( \overline{\Omega}_0^{End} \setminus \overline{\Omega}_0 \right) \subset \Sb^1.
\end{align*}
Since $\varphi$ is a homeomorphism, we see that $\Sb^1 \setminus A$ is a collection of arcs. We shall now show that the map $\varphi$ is a smooth embedding on $\Db \cup (\Sb^1 \setminus A)$.  It suffices to show that all first-order partial derivatives of $\varphi$ extend continuously to each $\xi \in (\Sb^1 \setminus A)$ and that $D\varphi(\xi)$ is non-singular. This follows from \cite[Theorem~3.5]{P1992}. We clarify that while the latter result, as stated in \cite{P1992}, requires $\varphi$ to be $\Cb$-valued, the arguments behind it just require $\varphi|_{\Db}$ to be a Riemann map and can be
localized to $\Db \cap U_{\xi}$, where $U_{\xi}$ is a neighborhood of $\xi$ ($\xi$ as above) such that $\overline{U}_{\xi} \cap (\Sb^1 \setminus A)$ is compact\,---\,see \cite{W1961} by Warschawski.

Next, consider $\Phi : \Bb_2 \rightarrow \Cb^2$ given by $\Phi(z_1,z_2) := (\varphi(z_1),z_2)$. We claim that $\Omega:=\Phi(\Bb_2)$ is a local Golidlocks domain. Notice that $\Phi$ extends to give a $\Cc^1$-smooth embedding
\begin{align*}
\partial\Bb_2 \setminus (A \times \{0\}) \rightarrow \Cb^2
\end{align*}
and 
\begin{align}\label{eq:bdy_good}
\partial\Omega = \Phi\big( \partial \Bb_2 \setminus (A \times \{0\}) \big).
\end{align}
Since, for each $x\in \partial \Bb_2 \setminus (A \times \{0\})$, there exists a neighborhood $V_x$ of $x$ on which (the extension of) $\Phi$ is a $\Cc^1$-diffeomorphism, we have the following estimates. There exists a neighborhood $\wt{V}_x$ of $x$ such that $\wt{V}_x \Subset V_x$, and a constant $C = C(x) > 1$ such that
\begin{align*}
  (1/C)\,\delta_{\Omega}(\Phi(z)) &\leq \delta_{\Bb_2}(z) \leq C\,\delta_{\Omega}(\Phi(z)), \text{ and } \\
  (1/C)\norm{\Phi^\prime(z)v} &\leq \|v\| \leq C\norm{\Phi^\prime(z)v},
\end{align*}
for all $z\in (\Bb_2 \cap \wt{V}_x)$ and for all $v \in \Cb^2$. Then, it follows from \eqref{eq:bdy_good} and the fact that $\Bb_2$ is a Goldilocks domain that each point in $\partial\Omega$ admits a neighborhood corresponding to which\,---\,in view of the estimates above\,---\,the two conditions in Definition~\ref{defn:local_GL} are satisfied. Thus, each point in $\partial\Omega$ is a local Goldilocks point. Furthermore, by construction, $\overline{\Omega}$ has uncountably many ends. \hfill $\blacktriangleleft$

\subsection{Domains with many boundary singularities} 
\begin{example} There exists a bounded domain $\Omega \subset \Cb^2$ with a subset $S \subset \partial \Omega$ having the following properties: 
\begin{itemize} 
\item $S$ is uncountable and totally disconnected,
\item $\partial \Omega$ is not $\Cc^1$-smooth at any point in $S$, 
\item every point in $\partial \Omega \setminus S$ is a local Goldilocks point. 
\end{itemize} 
\end{example}

\noindent Fix a Jordan curve $J \subset \Cb$ that is not $\Cc^1$-smooth at an uncountable totally disconnected set $J^\prime \subset J$ and where $J \setminus J^\prime$ consists of $\Cc^\infty$ arcs. Then let $\Omega_0 \subset \Cb$ denote the bounded component of $\Cb \setminus J$. One way to construct such an example is to let $J^\prime$ be the standard $\frac{1}{3}$-Cantor set in the real line, then let $\Omega_0 \subset \mathcal{H}:=\{z \in \Cb : {\rm Im}(z) > 0\}$ denote the open hyperbolic convex hull of $J^\prime$: i.e., the smallest open set in $\mathcal{H}$ that is geodesically convex with respect to the hyperbolic metric on $\mathcal{H}$ and whose closure contains $J^\prime$. 

Fix a biholomorphism $\varphi: \Db \rightarrow \Omega_0$. By Carath\'eodory's extension theorem, $\varphi$ extends to a continuous homeomorphism $\overline{\Db} \rightarrow \Omega_0 \cup J$, which we shall also call $\varphi$. Let 
\begin{align*}
A := \varphi^{-1}\left( J^\prime \right) \subset \Sb^1.
\end{align*}
Arguing as in the last example, $\varphi$ is a smooth embedding on $\Db \cup (\Sb^1 \setminus A)$.

Next, consider $\Phi : \Bb_2 \rightarrow \Cb^2$ given by $\Phi(z_1,z_2) := (\varphi(z_1),z_2)$. We claim that $\Omega:=\Phi(\Bb_2)$ has the desired properties. Notice that $\Phi$ extends to give a $\Cc^0$ embedding $\partial \Bb_2 \rightarrow \Cb^2$ which restricts to a $\Cc^1$-smooth embedding on $\partial\Bb_2 \setminus (A \times \{0\})$. We shall also use $\Phi$ to denote this extension. Then let
$$
S  := \Phi( A \times \{0\}) = J^\prime \times \{ 0\}. 
$$
Then $S$ is uncountable and totally disconnected. Further, since 
$$
\partial \Omega \cap (\Cb \times \{0\})= \Phi( \partial \Bb_2) \cap (\Cb \times \{0\}) = J \times \{0\},
$$
$\partial \Omega$ is not $\Cc^1$-smooth at any point in $S$. Finally, since for each $x\in \partial \Bb_2 \setminus (A \times \{0\})$ there exists a neighborhood $V_x$ of $x$ on which (the extension of) $\Phi$ is a $\Cc^1$-diffeomorphism, each point in $\partial \Omega \setminus S$ is a local Goldilocks point (which follows by the same argument as in the last example).  \hfill $\blacktriangleleft$

\subsection{Convex domains}
Known results allow us to identify two classes of convex domains that are locally Goldilocks domains. We first begin with a rather general result.

\begin{proposition}
If $\Omega \subset \Cb^d$ is a Kobayashi hyperbolic convex domain and $(\Omega, K_\Omega)$ is Gromov hyperbolic, then $\Omega$ is a local Golidlocks domain.
\end{proposition} 

The above follows from Theorem 6.1 in~\cite{Z2019}.  

In certain cases, the requirement that $(\Omega, K_\Omega)$ is Gromov hyperbolic too can be relaxed.

\begin{example}
There exist Kobayashi hyperbolic convex domains in $\Cb^d$, $d \geq 2$, such that $(\Omega, K_\Omega)$ is \textbf{not} Gromov hyperbolic, but $\Omega$ is locally a Goldilocks domain.
\end{example}

\noindent{By Corollary 1.13 in~\cite{Z2019}, if $\Bb_d \subset \Rb^d$ denotes the unit ball and $\Omega:= \Bb_d +i\Rb^d \subset \Cb^d$ then $(\Omega, K_\Omega)$ is not Gromov hyperbolic, but one can easily show that $\Omega$ is locally a Goldilocks domain. More generally, if $C \subset \Rb^d$ is a bounded convex domain with real analytic boundary, then $C+i\Rb^d$ is locally a Goldilocks domain for which the Kobayashi distance is not Gromov hyperbolic. \hfill $\blacktriangleleft$}

\section{Metrical preliminaries}
\subsection{The Kobayashi distance and metric}
In this section, we state a few facts on the connections between the Kobayashi pseudo-distance and the infinitesimal Kobayashi pseudo-metric. It is classical (and follows from the definition) that the function $k_\Omega$ is upper-semicontinuous with respect to the usual topology on $\Omega\times \Cb^d$. Thus, if a curve $\sigma: [a,b] \rightarrow \Omega$ is absolutely continuous (as a map $[a,b] \rightarrow \Cb^d$) then the function $[a,b]\ni t \mapsto k_\Omega(\sigma(t); \sigma^\prime(t))$ is integrable. So, we can define the \emph{length of $\sigma$ with respect to the Kobayashi pseudo-metric} as follows:
\begin{equation}\label{eq:k-length}
\ell_\Omega(\sigma)= \int_a^b k_\Omega(\sigma(t); \sigma^\prime(t)) dt.
\end{equation}

When $\Omega$ is Kobayashi hyperbolic, in which case the Kobayashi pseudo-distance $K_{\Omega}$ is a true distance, $k_{\Omega}$ has the following connections to $K_{\Omega}$:

\begin{result}\label{res:integ_dist}
Let $\Omega \subset \Cb^d$ be a Kobayashi hyperbolic domain.
\begin{enumerate}
\item\label{item:dist_integ_smooth} \cite[Theorem 1]{R1971} For any $z, w \in \Omega$ we have
\begin{multline*}
 K_\Omega(z,w) = \inf \left\{\ell_\Omega(\sigma)\,:\,\sigma\!:\![a,b]
 \rightarrow \Omega \text{ is piecewise } \Cc^1,\right. \\
 \left. \text{ with } \sigma(a)=z, \text{ and } \sigma(b)=w\right\}.
\end{multline*}
\item\label{item:dist_integ_absoc} \cite[Theorem 3.1]{V1989} For any $z, w \in \Omega$ we have
\begin{multline*}
 K_\Omega(z,w) = \inf \left\{\ell_\Omega(\sigma) : \sigma\!:\![a,b]
 \rightarrow \Omega \text{ is absolutely continuous}, \right. \\
 \left. \text{ with } \sigma(a)=z, \text{ and } \sigma(b)=w\right\}.
\end{multline*}
\end{enumerate}
\end{result}

\begin{result}[paraphrasing {\cite[Theorem~2]{R1971}}]\label{res:k-hyper}
Let $\Omega\subset \Cb^d$ be a domain. Then, $\Omega$ is Kobayashi hyperbolic if and only if for
each $z\in \Omega$, there exists a neighborhood $U_z$ of $z$ and a constant $c_z>0$ such that
$k_{\Omega}(w; v)\geq c_z\|v\|$ for every $w\in U_z$ and every $v\in T_w^{1,0}\Omega\cong \Cb^d$.
\end{result}

\subsection{The Hopf--Rinow theorem}
There is a different aspect of the Kobayashi distance that we shall require in this paper. We need a definition that makes sense for any metric space. Given a metric space $(X,d)$, the \emph{length of a continuous curve} $\sigma:[a,b] \rightarrow X$ is defined as
\[
l_d(\sigma) := \sup\left\{ \sum_{i=1}^n d(\sigma(t_{i-1}), \sigma(t_i) ): a = t_0 < t_2 < \dots < t_n=b\right\}.
\]
This gives the induced metric $d_I$ on $X$:
\[
d_I(x,y) = \inf\left\{ l_d(\sigma) : \sigma\!:\![a,b] \rightarrow X \text{ is continuous},
\sigma(a)=x, \text{ and } \sigma(b)=y\right\}.
\]
When $d_I = d$, the metric space $(X,d)$ is called a \emph{length metric space}. For such metric spaces, we have the following characterization of Cauchy completeness:
 
\begin{result}[Hopf--Rinow]\label{res:H-R}
Let $(X,d)$ be a length metric space, and suppose $(X, d)$ is locally compact. Then, the following are equivalent:
\begin{enumerate}
\item $(X,d)$ is a proper metric space: i.e., each bounded set in $X$ is relatively compact. 
\item $(X,d)$ is Cauchy complete. 
\end{enumerate}
Each of the above properties implies that $(X, d)$ is a geodesic space.
\end{result}

See \cite[Chapter~I.3]{BH1999} for a proof of the above version of the Hopf--Rinow theorem.

When a domain $\Omega \subset \Cb^d$ is Kobayashi hyperbolic, the metric space $(\Omega, K_\Omega)$ is a length metric space (by the definition of $K_{\Omega}$). Also, the topology induced by Kobayashi distance is the standard topology on $\Omega$ and so $(\Omega, K_\Omega)$ is locally compact. Thus, from Result~\ref{res:H-R}, we conclude:

\begin{proposition}\label{prop:Hopf_Rinow}
Suppose $\Omega \subset \Cb^d$ is a Kobayashi hyperbolic domain. Then the following are equivalent:
\begin{enumerate}
\item $(\Omega, K_\Omega)$ is a proper metric space. 
\item $(\Omega,K_\Omega)$ is Cauchy complete. 
\end{enumerate}
Each of the above properties implies that $(\Omega, K_{\Omega})$ is a geodesic space.
\end{proposition}

\subsection{The Gromov boundary}\label{ssec:Gromov_top}
We have assumed in the discussion in Sections~\ref{sec:intro} and~\ref{sec:examples} that the reader is familiar with Gromov hyperbolic metric spaces. However, a brief discussion on the Gromov boundary of a Gromov hyperbolic metric space is in order. Since, for the reasons mentioned in Section~\ref{sec:intro}, we shall not assume that\,---\,given a Kobayashi hyperbolic domain $\Omega \subset \Cb^d$\,---\,$(\Omega, K_{\Omega})$ is Cauchy-complete, the description of the Gromov boundary will be through Gromov sequences (defined below). Even in Section~\ref{ssec:Gromov_visi}, where the domains $\Omega\subset \Cb^d$ turn out to be Kobayashi complete, we cannot initially assume $\Omega$ to be Kobayashi complete. Thus, we shall avoid any mention of geodesic rays in connection with the Gromov boundary. We shall follow the treatment of the Gromov boundary presented in \cite{DSU2017}.

\begin{definition}\label{defn:G_seq}
Let $(X,d)$ be a metric space and, for $x, y, o \in X$, let $\grp{x}{y}{o}$ denote the Gromov product, which is the quantity $\big(d(o, x) + d(o, y) - d(x, y)\big)/2$.
\begin{enumerate}
\item A sequence $(x_n)_{n\geq 1}$ is called a \emph{Gromov sequence} if for some (and hence any) $o\in X$,
$$
  \lim_{m, n\to \infty}\grp{x_m}{x_n}{o} = +\infty.
$$
\item Let $(X, d)$ be Gromov hyperbolic. Two Gromov sequences $(x_n)_{n\geq 1}$ and $(y_n)_{n\geq 1}$ are said to be \emph{equivalent} if for some (and hence any) $o\in X$,
$$
  \lim_{m, n\to \infty}\grp{x_m}{y_n}{o} = +\infty.
$$
That the above relation between $(x_n)_{n\geq 1}$ and $(y_n)_{n\geq 1}$ is an equivalence relation follows from Gromov hyperbolicity of $(X, d)$.
\end{enumerate}
\end{definition}

For the remainder of this section, we shall assume that $(X,d)$ is Gromov hyperbolic. We shall denote the equivalence class of a Gromov sequence $(x_n)_{n\geq 1} \subset X$ by $[(x_n)_{n\geq 1}]$. As a \textbf{set}, the \emph{Gromov boundary} of $X$ is the set of all equivalence classes of Gromov sequences in $X$. As in Section~\ref{sec:intro}, we shall denote the Gromov boundary of $X$ by $\partial_{G}X$.

To describe the topology on $(X \cup \partial_{G}X)$, one must discuss how, given a base-point $o \in X$, $\grp{\bcdot}{\bcdot}{o} : X\times X \to \Rb$ extends  to $(X \cup \partial_{G}X)\times (X \cup \partial_{G}X)$. Since we will not need to explicitly understand the open sets of this topology, we shall not dwell on this extension (the interested reader is referred to \cite[Section~3.4]{DSU2017}). Instead, we present the following result that highlights some features of the topology on $(X \cup \partial_{G}X)$ that are relevant to this paper.

\begin{result}\label{res:Gromov_comptn}
Let $(X,d)$ be a Gromov hyperbolic metric space. Then:
\begin{enumerate}
\item\label{item:compact} If $(X, d)$ is proper and a geodesic space, then $(X \cup \partial_{G}X)$ and $\partial_{G}X$ are compact.
\item\label{item:convg} A sequence $(x_n)_{n\geq 1}$ in $X$ converges to a point $\xi \in \partial_{G}X$ if and only if $(x_n)_{n\geq 1}$ is a Gromov sequence and $[(x_n)_{n\geq 1}] = \xi$.
\end{enumerate}
\end{result}

\section{The end compactification of $\overline\Omega$}\label{sec:end_comp}

We refer the reader to \cite[Chapter~1]{S1970} for the definition of the end compactification of a topological space. Since, for a domain $\Omega\subset \Cb^d$, $\overline\Omega$ admits an increasing sequence of compact subsets $(K_n)_{n \geq 1}$ whose interiors relative to $\overline\Omega$ cover $\overline\Omega$, we shall give the following equivalent definition of $\overline{\Omega}^{End}$, which supports better the intuition behind some of our proofs. Given $\Omega\subset \Cb^d$, let $(K_n)_{n \geq 1}$ be as described above. An \emph{end}, say $e$, of $\overline\Omega$ is a decreasing sequence of connected open sets $\{U^{(e)}_n\}_{n \geq 1}$ such that 
$$
U^{(e)}_n := \text{some connected component of $\overline\Omega\setminus K_n$,} \; \; n=1, 2, 3, \dots\,.
$$
The end compactification of $\overline\Omega$, as a set, is $\overline\Omega$ together with all its ends. In the topology of $\overline{\Omega}^{End}$, for each end $e$, the sets $U^{(e)}_n\cup \{e\}$, $n=1, 2, 3, \dots$, represent open neighborhoods of $e$. This construction does not depend on the choice of the sequence $(K_n)_{n \geq 1}$ (observe that given another increasing sequence of compact subsets $(K^\prime_n)_{n \geq 1}$ whose interiors relative to $\overline\Omega$ cover $\overline\Omega$, there is an obvious bijection between the set of ends given by $(K_n)_{n \geq 1}$ and the set of ends given by $(K^\prime_n)_{n \geq 1}$).    
  
We now present the key lemma of this section. It will be needed in the proof of Theorem~\ref{thm:visible}. For a domain $\Omega\subset \Cb^d$, write $X_{\Omega} = \partial\Omega\setminus \partial_{{ \rm lg}}\Omega$ and let $\clos{X_{\Omega}}$ denote the closure of $X_{\Omega}$ in $ \overline{\Omega}^{End}$.

\begin{lemma}\label{lem:X_disconnected}
If $X_{\Omega}$ is totally disconnected, then $\clos{X_{\Omega}}$ is totally disconnected.
\end{lemma}
\begin{proof}
We first note that, by definition, $X_{\Omega}$ is closed in $\overline\Omega$. Thus,
\begin{equation}\label{E:closed}
\clos{X_{\Omega}}\cap \Cb^d = X_{\Omega} = \overline{X}_{\Omega}.
\end{equation}
Given this and the fact that $X_{\Omega}$ is totally disconnected, it is trivial that given distinct points $x, y\in \clos{X_{\Omega}}\cap \Cb^d$, there exist disjoint open subsets $A, B$ of $\clos{X_{\Omega}}$ such that $x\in A$, $y\in B$, and $\clos{X_{\Omega}}=A\cup B$. The proof of this lemma will follow from two claims.
\medskip

\noindent{{\bf Claim~I.} {\em  Given $x\in X_{\Omega}$ and $r>0$, there exist disjoint open subsets $A, B$ of $\clos{X_{\Omega}}$ such that $x \in A$, $A \subset \Bb_d(x,r)$, and $\clos{X_{\Omega}}=A \cup B$.}} 
\smallskip

\noindent{Define
\begin{align*}
K := X_{\Omega} \cap \overline{\Bb_d(x,2r)} \setminus \Bb_d(x,r).
\end{align*}
Since $X_{\Omega}$ is totally disconnected, for each $y \in K$ there exist disjoint open subsets $A_y, B_y$ of $X_{\Omega}$ such that $y \in B_y$, $x \in A_y$, and $ X_{\Omega}=A_y \cup B_y $. Since $K$ is compact we can find $y_1,\dots, y_N\in K$ such that $K \subset \cup_{j=1}^N B_{y_j}$. Then let 
\begin{align*}
A := \Bb_d(x,r) \cap \bigcap_{j=1}^N A_{y_j} \quad \text{and} \quad
B := \left(\clos{X_{\Omega}} \setminus \overline{\Bb_d(x,r)}\right) \cup \bigcup_{j=1}^N B_{y_j}.
\end{align*}
These sets clearly have the properties stated in the conclusion of Claim~I.}
\medskip

\noindent{{\bf Claim~II.} {\em  Given distinct points $x, y\in \clos{X_{\Omega}}\setminus \Cb^d$, there exist disjoint open subsets $A, B$ of $\clos{X_{\Omega}}$ such that $x \in A$, $y\in B$, and $\clos{X_{\Omega}}=A \cup B$.}} 
\smallskip

\noindent{Since $x, y$ are ends, we can fix a compact subset $K_0 \subset \overline{\Omega}$ and disjoint open subsets $U,V$ of $\overline{\Omega}^{End}$ such that $x \in U$, $y\in V$, and $U \cup V = \overline{\Omega}^{End} \setminus K_0$. }

By Claim~I, for each $\xi \in K_0\cap \clos{X_{\Omega}} =K_0 \cap X_{\Omega}$ there exist disjoint open subsets $A_{\xi}, B_{\xi}$ of $\clos{X_{\Omega}}$ such that $\xi \in A_{\xi}$, $A_{\xi}$ is bounded, and $\clos{X_{\Omega}}=A_{\xi} \cup B_{\xi}$. Since $K_0 \cap \clos{X_{\Omega}}$ is compact we can find $\xi_1,\dots, \xi_N\in K_0 \cap \clos{X_{\Omega}} $ such that $K_0\cap \clos{X_{\Omega}} \subset \cup_{j=1}^N A_{\xi_j}$. Then let 
\begin{align*}
A :=\left( U\cap \clos{X_{\Omega}}\right) \cup \bigcup_{j=1}^N A_{y_j}\quad \text{and} \quad
B := \left(V \cap \clos{X_{\Omega}}\right) \cap \bigcap_{j=1}^N B_{z_j}.
\end{align*}
By construction, $A$ and $B$ have the properties stated in the conclusion of Claim~II.

\medskip

In view of \eqref{E:closed}, and of the discussion following it, Claims~I and~II imply that $\clos{X_{\Omega}}$ is totally disconnected.
\end{proof}

\section{Length minimizing curves}\label{sec:min_curves}
This section is dedicated to some facts about $(\lambda, \kappa)$-almost-geodesics (which we had defined in Section~\ref{sec:intro}).
To this end, we need a couple of preliminary lemmas.

The first lemma follows directly from  \eqref{eq:k-length} and the triangle inequality for $K_{\Omega}$.

\begin{lemma}\label{lem:restricted_l}
Let $\Omega \subset \Cb^d$ be a domain and $\sigma: [a,b] \rightarrow \Omega$ an absolutely continuous curve. If 
\begin{align*}
\ell_\Omega(\sigma) \leq K_\Omega(\sigma(a),\sigma(b)) + \epsilon
\end{align*}
then, whenever $a \leq a^\prime \leq b^\prime \leq b$, we have
\begin{align*}
\ell_\Omega(\sigma|_{[a^\prime, b^\prime]}) \leq K_\Omega(\sigma(a^\prime),\sigma(b^\prime)) + \epsilon.
\end{align*}
\end{lemma}

The second lemma that we need is as follows:

\begin{lemma}\label{lem:euclidean_estimates}
Let $\Omega \subset \Cb^d$ be a Kobayashi hyperbolic domain and let $\sigma: [0, T]\to \Omega$ be of class $\Cc^1$. Then, there exists a constant $C \geq 1$ such that
$$
\frac{1}{C} \norm{\sigma^\prime(t)} \leq k_\Omega(\sigma(t); \sigma^\prime(t)) \leq C \norm{\sigma^\prime(t)}
$$
for every  $t \in [0,T]$.
\end{lemma}

The upper bound is a consequence of a standard comparison argument based on the fact that there exists an $r>0$ such that $\Bb_d(z,r)\subset \Omega$ for every $z\in \sigma([0,T])$. The lower bound is an elementary consequence of compactness and Result~\ref{res:k-hyper}. 

With these lemmas, we can establish the principal result of this section:

\begin{proposition}\label{prop:almost_geod_exist}
Let $\Omega \subset \Cb^d$ be a Kobayashi hyperbolic domain. For any $\kappa > 0$ and $x,y \in \Omega$ there exists a $(1,\kappa)$-almost-geodesic $\sigma:[a,b] \rightarrow \Omega$ with $\sigma(a) =x$ and $\sigma(b)=y$. 
\end{proposition}

The proof of the above result is largely the same as the proof of \cite[Proposition~4.4]{BZ2017}, which relies on the conclusions of Result~\ref{res:integ_dist}. For one set of estimates for which one needs the boundedness of $\Omega$ in the proof of \cite[Proposition~4.4]{BZ2017}, we must now appeal to Lemma~\ref{lem:euclidean_estimates}. Since, barring the latter step, the proof of Proposition~\ref{prop:almost_geod_exist} is so similar to the proof of \cite[Proposition~4.4]{BZ2017}, we shall skip the proof.

\begin{proposition}\label{prop:limits_of_geodesic_rays} 
Let $\Omega \subset \Cb^d$ be a visibility domain. Let $\lambda \geq 1$ and $\kappa \geq 0$. Suppose $(\sigma_n)_{n \geq 1}$, $\sigma_n : [0,b_n] \rightarrow \Omega$, is a sequence of $(\lambda, \kappa)$-almost-geodesics converging locally uniformly to a $(\lambda, \kappa)$-almost-geodesic $\sigma:[0,\infty) \rightarrow \Omega$. Then the two limits below exist in $\partial\overline{\Omega}^{End}$ and 
$$
\lim_{t \rightarrow \infty} \sigma(t) = \lim_{n \rightarrow \infty} \sigma_n(b_n).
$$ 
\end{proposition}

\begin{proof}
We begin by showing that $\lim_{t \rightarrow \infty} \sigma(t)$ exists. Suppose this is not so. Since the Kobayashi distance is finite on compact sets, any limit point of $\sigma$ must be in $\partial\overline{\Omega}^{End}$. Then, as $\overline{\Omega}^{End}$ is compact, there exist sequences $(s_n)_{n \geq 1}$ and $(t_n)_{n \geq 1}$ in $[0, \infty)$ such that $s_n, t_n \rightarrow \infty$, and points $\xi_1, \xi_2\in \partial\overline{\Omega}^{End}$ such that 
$$
\lim_{n \rightarrow \infty} \sigma(s_n) = \xi_1 \neq \xi_2 = \lim_{n \rightarrow \infty} \sigma(t_n).
$$
By passing to a subsequence of $(t_n)_{n \geq 1}$ and relabeling if needed, we may assume that $t_n > s_n$ for $n=1, 2, 3, \dots$\,. By the visibility assumption 
$$
r:=\sup_{n \geq 1} K_{\Omega}\left(\sigma(0),  \left. \sigma \right|_{[s_n, t_n]}\right) <+\infty. 
$$
For each $n \geq 1$, let $\tau_n \in [s_n, t_n]$ be such that
$K_{\Omega}(\sigma(0), \sigma(\tau_n)) \leq r$. Then, by definition
$$
\lambda^{-1} \tau_n - \kappa \leq K_{\Omega}(\sigma (0), \sigma(\tau_n)) \leq  r
$$
for each $n \geq 1$, which contradicts the fact that $\tau_n \rightarrow +\infty$. Thus,
$\lim_{t \rightarrow \infty} \sigma(t) =: \xi$ exists in $\partial\overline{\Omega}^{End}$.

Suppose for a contradiction that either $\lim_{n \rightarrow \infty} \sigma_n(b_n)$ does not exist or $\xi \neq \lim_{n \rightarrow \infty} \sigma_n(b_n)$. In either case, there exist $\eta \in \overline{\Omega}^{End}$  and a subsequence $(\sigma_{n_j})_{j \geq 1}$ such that $\sigma_{n_j}(b_{n_j}) \rightarrow \eta$ and $\eta \neq \xi$. 

Let $(U_{\nu})_{\nu \geq 1}$ be a decreasing sequence of $\overline{\Omega}^{End}$-open neighborhoods of $\xi$ such that $\cap_{\nu \geq 1}\overline{U_{\nu}}=\{\xi\}$ (if $\xi \in \partial\overline{\Omega}^{End}\setminus \partial\Omega$, then the existence of $(U_{\nu})_{\nu \geq 1}$ follows from the discussion at the beginning of Section~\ref{sec:end_comp}). For each $\nu$, there exists $s_\nu^\prime$ such that $\sigma(t) \in U_{\nu}$ for every $t \geq s_\nu^\prime$. Further, by assumption, 
$$
\lim_{n \rightarrow \infty} \sigma_n(s_\nu^\prime) = \sigma(s_\nu^\prime). 
$$
So by replacing $(\sigma_{n_j})_{j \geq 1}$ with a subsequence, we can suppose that $s_j^\prime \leq b_{n_j}$ and that $\sigma_{n_j}(s_{j}^\prime) \in U_j$. Then, as $ \cap_{\nu \geq 1}\overline{U_{\nu}}=\{\xi\}$, $\sigma_{n_j}(s_{j}^\prime) \rightarrow \xi$. 

Write $I_j := [s_j^\prime, b_{n_j}]$, $j=1, 2, 3, \dots$\,. By the visibility assumption
$$
\sup_{j \geq 1} K_{\Omega}\left(\sigma(0), \left. \sigma_{n_j}\right|_{I_j}\right) <+\infty.
$$
Let $c_j \in I_j$ be such that
$$
K_{\Omega}(\sigma(0), \sigma_{n_j}(c_j)) = K_{\Omega}\left(\sigma(0), \left. \sigma_{n_j}\right|_{I_j}\right).
$$
By the same argument as in the proof that $\lim_{t \rightarrow \infty} \sigma(t)$ exists, we deduce that $\sup_{j \geq 1} c_j < +\infty$, which contradicts the facts that $s_j^\prime \leq c_j$ and $s_j^\prime \rightarrow +\infty$. By this contradiction we deduce that $\xi = \lim_{n \rightarrow \infty} \sigma_n(b_n)$.

\end{proof}

\begin{proposition}\label{prop:Lipschitz}
Let $\Omega \subset \Cb^d$ be a Kobayashi hyperbolic domain and suppose $x \in \partial_{{ \rm lg}}\Omega$. Let $U$ be a bounded $\overline{\Omega}$-open neighborhood of $x$ in which the two conditions given in Definition~\ref{defn:local_GL} hold true. For each $\lambda \geq 1$ there exists a $M = M(\lambda, U)>0$ such that any $(\lambda, \kappa)$-almost-geodesic $\sigma:I \rightarrow \Omega$ with image in $U\cap \Omega$ is $M$-Lipschitz with respect to the Euclidean distance. 
\end{proposition}
\begin{proof}
Since $x$ is a local Goldilocks point, it follows from Condition~\ref{item:metric} in Definition~\ref{defn:local_GL} that
$M_{\Omega, U}(r) \rightarrow 0$ as $r \rightarrow 0^+$. Let $\delta_0 > 0$ be such that
\[
0 < M_{\Omega, U}(r) < 1
\]
whenever $r \in (0, \delta_0)$. Then, by definition
\begin{equation}\label{eq:low_bound_1}
k_{\Omega}(z;v) \geq \norm{v} \quad \forall z \in U\cap \Omega \text{ such that $\delta_{\Omega}(z) < \delta_0$,
and $\forall v \in \Cb^d$.}
\end{equation}
Since $U\cap \{z \in \Omega : \delta_{\Omega}(z) \geq \delta_0 \} \Subset \Omega$, a standard comparison argument for $k_{\Omega}$ implies that there exists $c_0 > 0$ such that 
\begin{equation}\label{eq:low_bound_2}
k_{\Omega}(z;v) \geq c_0 \norm{v} \quad \forall z \in U\cap \{z \in \Omega : \delta_{\Omega}(z) \geq \delta_0 \},
\text{ and $\forall v \in \Cb^d$.}
\end{equation}
Write $c_1 := \min\{c_0, 1\}$. Now, fix $\lambda \geq 1$. We claim that every $(\lambda, \kappa)$-almost-geodesic with image in $U \cap \Omega$ is $\lambda/c_1$-Lipschitz (with respect to the Euclidean distance). 

Suppose that $\sigma:I \rightarrow \Omega$ is an $(\lambda, \kappa)$-almost-geodesic with image in $U \cap \Omega$. Then, by \eqref{eq:low_bound_1} and \eqref{eq:low_bound_2}, for almost every $t \in I$ we have
\begin{align*}
\norm{\sigma^\prime(t)} \leq \frac{1}{c_1} k_\Omega(\sigma(t);\sigma^\prime(t)) \leq \frac{\lambda}{c_1}.
\end{align*}
Since $\sigma$ is absolutely continuous we have
\begin{align*}
\sigma(t) = \sigma(s) + \int_s^{t} \sigma^\prime(r) dr.
\end{align*}
Thus 
\begin{align*}
\norm{\sigma(t) - \sigma(s)} = \norm{ \int_s^t \sigma^\prime(r)dr } \leq  \frac{\lambda}{c_1} \abs{t-s},
\end{align*}
whence $\sigma$ is $\lambda/c_1$-Lipschitz. 
\end{proof}

\section{The proof of Theorem~\ref{thm:visible}}

We need a couple of technical lemmas to prove Theorem~\ref{thm:visible}

\begin{lemma}\label{lem:nonconst}
Let $\Omega$ be a domain in $\Cb^d$ and suppose $x^0\in \partial_{{ \rm lg}} \Omega$. Let $U$ be an $\overline{\Omega}$-open neighborhood of $x^0$ in which the two conditions given in Definition~\ref{defn:local_GL} hold true. Let $\lambda \geq 1$ and $\kappa \geq 0$ be constants. Suppose $\gamma_n : [a_n, b_n] \rightarrow \Omega$ are $(\lambda, \kappa)$-almost-geodesics with image in $\Omega\cap U$, $\xi^\prime\neq \eta^\prime$ are distinct points in $U\cap \overline{\Omega}$, and $a \in [-\infty, 0]$, $b \in [0, +\infty]$ are such that $a<b$ and, furthermore, satisfy:
\begin{align*}
  a_n\to a &\text{ and } b_n\to b \; \text{ as $n\to \infty$}, \\
  \gamma_n(a_n)\to \xi^\prime &\text{ and } \gamma_n(b_n)\to \eta^\prime \; \text{ as $n\to \infty$}, \\
  (\gamma_n)_{n \geq 1} &\text{ converges locally uniformly on $(a, b)$ to a curve $\gamma: (a,b) \rightarrow \overline{\Omega}$}.
 \end{align*}
Then $\gamma$ is non-constant.
\end{lemma}

The above lemma is Claim~2 in the proof of \cite[Theorem~1.4]{BZ2017} with the one difference, which is immaterial to the proof, that as the Goldilocks conditions of Definition~\ref{defn:local_GL} are localized to $U$, the function $M_{\Omega,U}$ replaces the function $M_{\Omega}$ in the latter claim.

\begin{lemma}\label{lem:ae_bound}
Let $\Omega$, $x^0\in \partial_{{ \rm lg}} \Omega$, and $U\subset \overline{\Omega}$ be as in Lemma~\ref{lem:nonconst}. Suppose $\gamma: [A,B] \rightarrow U$ is a $(\lambda, \kappa)$-almost-geodesic of $\Omega$ for some $\lambda \geq 1$ and $\kappa \geq 0$. Then,
\begin{align*}
\norm{\gamma^\prime(t)} \leq \lambda M_{\Omega,U}(\delta_{\Omega}(\gamma(t)))
\end{align*}
for almost every $t\in [A,B]$.
\end{lemma}

The above inequality just follows from the definitions. A proof of it can be found within the proof of \cite[Theorem~1.4]{BZ2017} (with the one difference that the function $M_{\Omega}$ serves the role of $M_{\Omega,U}$ in the latter proof).

\begin{proof}[The proof of Theorem~\ref{thm:visible}]
Fix constants $\lambda \geq 1$ and $\kappa \geq 0$. Also fix $\xi,\eta \in \partial \overline{\Omega}^{End}$ and $V_\xi, V_\eta$ as in Definition~\ref{defn:visib_dom}. Fix a sequence of compact subsets $(\Ks_n)_{n \geq 1}$ of $\Omega$ such that 
\begin{align*}
\Ks_n & \subset \Ks^{\circ}_ {n+1}, \quad n=1,2,3,\dots, \text{ and } \\
\Omega &= \cup_{n \geq 1}\Ks_n.
\end{align*}
Assume, aiming for a contradiction, that for each $n \geq 1$ there exists a $(\lambda, \kappa)$-almost-geodesic $\sigma_n: [0,T_n] \rightarrow \Omega$ with $\sigma_n(0) \in V_\xi$, $\sigma_n(T_n) \in V_\eta$, and $\sigma_n \cap \Ks_n = \emptyset$. 

Since $C_n:=\sigma_n([0,T_n])$ is a compact subset of $\Cb^d$, passing to a subsequence if needed, we may suppose that $(C_n)_{n\geq 1}$ converges in the local Hausdorff topology to a closed subset $C \subset \Cb^d$. Since $\clos{V_\xi} \cap \clos{V_\eta} = \emptyset$, the set $C$ must be non-empty. Since $\sigma_n \cap \Ks_n = \emptyset$ for each $n$, we must have $C \subset \partial\Omega$. 

\begin{lemma}\label{lem:C_connected}
The closure of $C$ in $ \overline{\Omega}^{End}$ is connected.
\end{lemma} 

\begin{proof} Suppose, for a contradiction, that $\clos{C}$, the closure of $C$ in $ \overline{\Omega}^{End}$, is not connected. Then there exist disjoint open sets $\Oc_1, \Oc_2 \subset  \overline{\Omega}^{End}$ such that $\clos{C}$ intersects both $\Oc_1, \Oc_2$, and $\clos{C} \subset \Oc_1 \cup \Oc_2$. Then for $n$ sufficiently large $\sigma_n([0,T_n])$ intersects both $\Oc_1$ and $\Oc_2$. Since $\sigma_n([0,T_n])$ is connected, for every $n$ sufficiently large there exists 
\begin{align*}
z_n \in \sigma_n([0,T_n]) \setminus (\Oc_1 \cup \Oc_2).
\end{align*}
We can find a subsequence $(z_{n_j})_{j\geq 1}$ such that $z_{n_j} \rightarrow z_\infty \in  \overline{\Omega}^{End}$. Then 
\begin{align*}
z_\infty \in \clos{C} \setminus (\Oc_1 \cup \Oc_2),
\end{align*}
and we have a contradiction. Hence $\clos{C}$ is connected. 
\end{proof}

From Lemmas~\ref{lem:X_disconnected} and~\ref{lem:C_connected}, and the fact that $\clos{V_{\xi}}\cap\clos{V_{\eta}} = \emptyset$, there exists $x^0 \in C \cap \partial_{{ \rm lg}} \Omega$. Fix a bounded neighborhood $\Oc$ of $x^0$ that satisfies the Definition~\ref{defn:local_GL}. Recall that $C$ is the limit in the local Hausdorff topology of $(\sigma_n[0, T_n])_{n\geq 1}$. Replacing $(\sigma_n[0, T_n])_{n\geq 1}$ by a subsequence, if needed, we can find, for each $n \geq 1$
\begin{itemize}
\item $t_n\in [0, T_n]$, and
\item an interval $[a_n,b_n] \subset [0,T_n]$ containing $t_n$,
\end{itemize}
such that $\sigma_n(t_n) \rightarrow x^0$, $\sigma_n([a_n,b_n]) \subset \Oc$ for every $n\geq 1$, and
\begin{equation}\label{eq:separated}
\inf_{n \geq 1} \norm{\sigma_n(b_n)-\sigma_n(a_n)} > 0.
\end{equation}
Since $\Oc$ is bounded, there exists an $R>0$ such that
\begin{equation}\label{eq:finite}
\sigma_n([a_n, b_n]) \subset \Bb_d(0,R) \quad \text{for each} \quad n \geq 1.
\end{equation}
By an affine reparametrization of each $\sigma_n$, we may assume that $0\in [a_n, b_n]$ (we will continue to refer to
the reparametrizations as $\sigma_n$) and
\begin{align*}
\delta_\Omega(\sigma_n(0)) = \max\{ \delta_\Omega(\sigma_n(t)) : t \in [a_n,b_n]\}.
\end{align*}
Then, by passing to a subsequence if needed, we can assume $a_n \rightarrow a \in [-\infty,0]$, $b_n \rightarrow b \in [0,\infty]$, $\sigma_n(a_n) \rightarrow \xi^\prime$, and $\sigma_n(b_n) \rightarrow \eta^\prime$. By \eqref{eq:separated}, there exists a constant $\epsilon>0$ such that 
\begin{align*}
\norm{\xi^\prime -\eta^\prime} \geq \epsilon.
\end{align*}

By Proposition~\ref{prop:Lipschitz}, there exists a $M>0$ such that each $\sigma_n$ is $M$-Lipschitz with respect
to the Euclidean distance. Given this fact and the boundedness condition \eqref{eq:finite} we conclude, passing to a further subsequence and relabeling as $(\sigma_n)_{n \geq 1}$, that $\sigma_n|_{(a_n, b_n)}$ converges locally uniformly on $(a,b)$ to a  curve $\sigma:(a,b) \rightarrow \overline{\Omega}$ (we restrict to the open interval because $a$ could be $-\infty$ and $b$ could be $\infty$). Notice that $a \neq b$ because each $\sigma_n$ is $M$-Lipschitz and so
\begin{align*}
0 < \norm{\xi^\prime-\eta^\prime} \leq M\abs{b-a}.
\end{align*}
All the conditions of Lemma~\ref{lem:nonconst}, taking $\gamma_n = \sigma_n|_{[a_n, b_n]}$ and $U = \Oc$, hold true. Thus, $\sigma$ is non-constant.

Recall that 
\begin{align*}
\delta_\Omega(\sigma_n(t)) \leq \delta_\Omega(\sigma_n(0))
\end{align*}
for each $t\in [a_n, b_n]$. Furthermore, as $\sigma_n([a_n, b_n])\subset \Oc$, we have 
\begin{align*}
M_{\Omega, \Oc}(\delta_\Omega(\sigma_n(t))) \leq M_{\Omega, \Oc}(\delta_\Omega(\sigma_n(0)))
\end{align*}
for every $t\in [a_n, b_n]$. Thus $M_{\Omega, \Oc}(\delta_\Omega(\sigma_n(t))) \rightarrow 0$. But now, if $s \leq u$ and $s,u \in (a,b)$ then:
\begin{align*}
 \norm{\sigma(s)-\sigma(u)} 
 = \lim_{n \rightarrow \infty} \norm{\sigma_n(s)-\sigma_n(u)}
 &\leq \limsup_{n \rightarrow \infty} \int_s^u \norm{\sigma_n^\prime(t)} dt\\
 & \leq  \lambda\limsup_{n \rightarrow \infty} \int_u^w\!\!M_{\Omega, \Oc}(\delta_\Omega(\sigma_n(t))) dt = 0.
\end{align*}
The last inequality follows by applying Lemma~\ref{lem:ae_bound} to each $\sigma_n|_{[a_n, b_n]}$
Thus $\sigma$ is constant. But this contradicts the conclusion of the last paragraph. Thus, our assumption above must be false; hence the result.
\end{proof}

\section{Boundary extensions: the proof of Theorem~\ref{thm:extensions}}\label{sec:quasi_isom_extn}
To prove Theorem~\ref{thm:extensions}, we first need to explain one of the terms featured in its hypothesis. Given a domain $\Omega \subset \Cb^d$, we say that \emph{$\Omega$ has well-behaved geodesics} if $\Omega$ is complete Kobayashi hyperbolic (in which case $(\Omega, K_{\Omega})$ is a geodesic space\,---\,see Proposition~\ref{prop:Hopf_Rinow}) and, whenever $(z_n)_{n \geq 1}, (w_n)_{n \geq 1}$ are sequences in $\Omega$ with $\lim_{n \rightarrow \infty} w_n = \lim_{n \rightarrow \infty} z_n=\xi \in \overline{\Omega}^{End} \setminus \Omega$ and $\sigma_n$ is a sequence of geodesic segments joining $w_n$ to $z_n$, we have
$$
\lim_{n \rightarrow \infty} K_\Omega(o, \sigma_n) = +\infty
$$
for some (hence any) $o \in \Omega$.

We also need one lemma. Before stating it, a definition: given constants $\lambda \geq 1$ and $\kappa \geq 0$, a map $\sigma : I \rightarrow \Omega$ of an interval $I \subset \Rb$ is called a \emph{$(\lambda, \kappa)$-quasi-geodesic} if 
$$
\lambda^{-1}\abs{t-s} - \kappa \leq K_\Omega(\sigma(t), \sigma(s)) \leq \lambda\abs{t-s} + \kappa
$$ 
for all $s,t \in I$. In other words, a $(\lambda, \kappa)$-quasi-geodesic is a map that satisfies \textbf{only} the first of the two conditions defining a $(\lambda, \kappa)$-almost-geodesic (see Section~\ref{sec:intro}). So, for $\kappa > 0$, $(\lambda, \kappa)$-quasi-geodesics are not necessarily continuous.

\begin{lemma}\label{lem:approximating_quasi-geodesics} 
Let $\Omega \subset \Cb^d$ be Kobayashi hyperbolic. For any $\lambda_0 \geq 1, \kappa_0 \geq 0$ there exist constants $\lambda \geq 1$, $\kappa \geq 0$, and $r > 0$, depending only on the pair $(\lambda_0, \kappa_0)$ such that the following holds: if $\sigma : [a,b] \rightarrow \Omega$ is a $(\lambda_0,\kappa_0)$-quasi-geodesic, then there exists a $(\lambda,\kappa)$-almost-geodesic $\hat{\sigma} : [\hat{a},\hat{b}] \rightarrow \Omega$ with $\hat{\sigma}(\hat{a})=\sigma(a)$, $\hat{\sigma}(\hat{b})=\sigma(b)$, and such that
\begin{align*}
K_\Omega^{{\rm Haus}}(\sigma, \hat{\sigma}) \leq r.
\end{align*}
Here $K_\Omega^{{\rm Haus}}(A, B)$ denotes the Hausdorff distance with respect to $K_\Omega$ between the sets $A, B\subset \Omega$. 
\end{lemma} 

\begin{proof}[Proof sketch] The idea is to pick finitely many points
$$
a = t_0 < t_1 < \dots < t_N = b,
$$
these points being suitably chosen, and let $\hat{\sigma}$ be the curve obtained by joining $\sigma(t_{j-1})$ and $\sigma(t_{j})$ by a $(1,1)$-almost-geodesic, $j=1, \dots, N$; see the proof of \cite[Proposition 4.9]{BZ2017} for details.

The latter proposition differs from Lemma~\ref{lem:approximating_quasi-geodesics} in only one respect: here, $\Omega$ is not necessarily bounded. So, we appeal to Proposition~\ref{prop:almost_geod_exist} for the existence of the above-mentioned $(1,1)$-almost-geodesics. In all other aspects, the present proof is the same as that of \cite[Proposition 4.9]{BZ2017}.  
\end{proof} 

We can now provide

\begin{proof}[The proof of Theorem~\ref{thm:extensions}] We first show that $f$ extends to a naturally-defined map $\overline{\Omega}^{End}_1 \rightarrow\overline{\Omega}^{End}_2$. Fix some $\xi \in \overline{\Omega}^{End}_1\setminus \Omega_1$. We claim that $\lim_{z \rightarrow \xi} f(z)$ exists in $\overline{\Omega}_2^{End}$. Suppose not. Then there exist sequences $(w_n)_{n \geq 1}$ and $(z_n)_{n \geq 1}$ in $\Omega_1$ where $\lim_{n \rightarrow \infty} z_n = \lim_{n \rightarrow \infty} w_n = \xi$ but 
$$
\lim_{n \rightarrow \infty} f(z_n) = \eta_1 \neq \eta_2 = \lim_{n \rightarrow \infty} f(w_n).
$$
Since $f$ is a quasi-isometric embedding, $\eta_1, \eta_2 \in \overline{\Omega}^{End}_2 \setminus \Omega_2$. 

For each $n$, let $\sigma_n$ be a geodesic joining $w_n$ to $z_n$ (which exists since $\Omega_1$ is a complete Kobayashi hyperbolic domain). Then each $f \circ \sigma_n$ is a quasi-geodesic in $\Omega_1$ with constants that are independent of $n$. Fix $b_0 \in \Omega_1$. Since $\Omega_2$ is a visibility domain, Lemma~\ref{lem:approximating_quasi-geodesics} implies that 
$$
\sup_{n \geq 1} K_{\Omega_2}(f(b_0), f \circ \sigma_n) <+\infty. 
$$
Hence 
$$
\sup_{n \geq 1} K_{\Omega_1}(b_0,  \sigma_n) <+\infty,
$$
and we have a contradiction. Hence $f_\infty(\xi) : = \lim_{{\Omega_1}\ni z \rightarrow \xi} f(z)$ exists for every $\xi \in  \overline{\Omega}^{End}_1\setminus\Omega_1$. 

Define
$$
\wt{f}(x) = \begin{cases}
 			f(x), &\text{if $x \in \Omega_1$}, \\
 			f_{\infty}(x), &\text{if $x \in \partial\overline{\Omega_1}^{End}$}.
 			\end{cases}
$$
We claim that $\wt{f}$ is continuous. By the definition of $\wt{f}$, it is easy to see that it suffices to show that $f_\infty : \partial\overline{\Omega}^{End}_1 \rightarrow \partial\overline{\Omega}^{End}_2$ is continuous. Since $\partial\overline{\Omega}^{End}_2$ is compact, it is enough to fix a sequence $(\xi_n)_{n \geq 1}$ in $\partial\overline{\Omega}^{End}_1\setminus \{\xi\}$ where $\xi_n \rightarrow \xi$, $f_\infty(\xi_n) \rightarrow \eta$, and then show that $\eta = f_\infty(\xi)$. Since $\lim_{z \rightarrow \xi_n} f(z)=f_\infty(\xi_n)$, for each $n$ we can find $z_n \in \Omega_1$ sufficiently close to $\xi_n$ such that $z_n \rightarrow \xi$ and $f_\infty(z_n) \rightarrow \eta$. Then 
$$
\eta = \lim_{n \rightarrow \infty} f_\infty(z_n) = \lim_{{\Omega_1}\ni z \rightarrow \xi} f(z)=f_\infty(\xi)
$$
and we are done.
\end{proof}

\section{Kobayashi geometry of planar domains: proof of Theorem~\ref{thm:planar_domains}}
This section and the next will focus on planar domains. The results in these sections are, essentially, consequences of the behavior of the Kobayashi distance and the Kobayashi metric. This section is dedicated to the proof of Theorem~\ref{thm:planar_domains} while the next section will deal with the examples mentioned in Section~\ref{ssec:applications_maps}.

The proof of Theorem~\ref{thm:planar_domains} will require some setting up and a few preliminary lemmas. Fix $\Omega_1, \Omega_2 \subset \Cb$ as in the statement of the theorem.

Since $\Omega_1$ and $\Omega_2$ are Lipschitz domains, for each $\xi \in\partial\Omega_j$ there exist  $\nu_\xi \in \Sb^1$, $\theta_\xi > 0$, $r_\xi > 0$, and a neighborhood $\Oc_\xi \subset \Cb$ of $\xi$ with the following property (recall our notation from Section~\ref{sec:examples} and, in particular, Remark~\ref{rem:about_Lip_dom}): if $x \in \partial \Omega_j \cap \Oc_\xi$, then
$$
(x+\Gamma_x(\nu_\xi, \theta_\xi)) \cap \Bb(x,r_\xi) \subset \Omega_j
$$ 
and 
$$
(x+\Gamma_x(-\nu_\xi, \theta_\xi)) \cap \Bb(x,r_\xi) \subset (\Cb \setminus \Omega_j).
$$
The symbols $\nu_\xi$ and $r_\xi$ and $\theta_\xi$ that will appear, without any further explanation, in the proofs of the lemmas below are as defined above.  

\begin{lemma}\label{lem:lower_bd_planar_ext_pf}
For each $\xi \in \partial \Omega_j$ there exists $a_\xi > 0$ such that 
$$
k_{\Omega_j}(z; v) \geq a_\xi\frac{\abs{v}}{\abs{z-x}}
$$
for every $x \in \Oc_\xi \cap \partial\Omega_j$ and every $z \in \Bb(x, r_\xi/2) \cap \Omega_j$. 
\end{lemma} 

\begin{proof} Fix $j\in \{1, 2\}$ and $\xi \in \partial \Omega_j$. Notice that there exists $m=m(\xi) >0$ such that if $x \in \Oc_\xi \cap \partial \Omega_j$ and $t \in (0,r_\xi/2)$, then $\Bb(x-t \nu_\xi, m t) \subset \Cb \setminus \Omega_j$. 

Fix $x \in \Oc_\xi \cap \partial\Omega_j$ and $z \in \Bb(x, r_\xi/2) \cap \Omega$. Set $w := x - \abs{z-x}\nu_\xi$ and consider  the holomorphic embedding $f : \Omega_j \rightarrow \Db$ given by 
$$
f(\zeta) := \frac{m \abs{z-x}}{\zeta-w}, \quad \zeta \in \Omega_j.
$$
The conclusion now follows by the same argument as in the proof of Lemma~\ref{lem:ext_cone}.
\end{proof}

\begin{proposition}\label{prop:finding_quasi_geodesics}
For each $\xi \in \partial \Omega_j$ there exists a constant $\lambda_\xi \geq 1$ such that if $x \in \partial \Omega_j \cap \Oc_\xi$, then the map $\sigma_{\xi, x} : [0,\infty) \rightarrow \Omega_j$ given by
$$
\sigma_{\xi,x}(t) := x + \frac{r_\xi}{2}e^{-2t}\nu_\xi, \quad t \in [0, \infty),
$$
is a $(\lambda_\xi,0)$-almost-geodesic. 
\end{proposition}

\begin{proof} Fix $j\in \{1, 2\}$ and $\xi \in \partial \Omega_j$. Notice that there exists $m=m(\xi) >0$ such that if $x \in \Oc_\xi \cap \partial \Omega_j$ and $t \in (0,r_\xi/2)$, then
\begin{align}
\Bb(x+t \nu_\xi, m t) &\subset \Omega_j, \text{ and } \label{eq:ball_in} \\
\Bb(x-t \nu_\xi, m t) &\subset \Cb \setminus \Omega_j. \notag
\end{align}
Next, fix $x \in \Oc_{\xi} \cap \partial \Omega_j$ and $0 < t_1 < t_2$. Then 
\begin{align*}
K_{\Omega_j}( \sigma_{\xi, x}(t_1),  \sigma_{\xi, x}(t_2))
  &\leq \int_{t_1}^{t_2} k_{\Omega_j}(\sigma_{\xi,x}(s); \sigma_{\xi,x}^\prime(s)) ds \\
  &\leq \int_{t_1}^{t_2} \frac{\big|\sigma_{\xi,x}^\prime(s)\big|}{\delta_{\Omega_j}(\sigma_{\xi,x}(s))}ds
    \leq \int_{t_1}^{t_2} \frac{2}{m}ds = \frac{2}{m} (t_2-t_1).
\end{align*}
The second inequality follows by a comparison of the Kobayashi metrics of $\Omega_j$ and \linebreak
$\Bb\big(\sigma_{\xi,x}(s), \delta_{\Omega_j}(\sigma_{\xi,x}(s))\big)$ and the fact that the inclusion $\Bb\big(\sigma_{\xi,x}(s), \delta_{\Omega_j}(\sigma_{\xi,x}(s))\big) \hookrightarrow \Omega_j$ is holomorphic. The third inequality follows by computing $\big|\sigma_{\xi,x}^\prime(s)\big|$ and observing that, due to \eqref{eq:ball_in},
\[
  \frac{mr_{\xi}}{2}e^{-2s} \leq \delta_{\Omega_j}(\sigma_{\xi,x}(s)).
\]
    
For the lower bound, let $\sigma : [0,1] \rightarrow \Omega_j$ be an absolutely continuous curve with $\sigma(0)=\sigma_{\xi, x}(t_1)$ and $\sigma(1) = \sigma_{\xi, x}(t_2)$. Let 
$$
a := \max\left\{ t \in [0,1] : \abs{\sigma(t)-x} \geq  \frac{r_\xi}{2} e^{-2t_1} \right\}.
$$ 
Then 
$$
\abs{\sigma(t)-x} < \frac{r_\xi}{2} e^{-2t_1} < \frac{r_\xi}{2}
$$
for all $t \in (a,1]$. So, by Lemma~\ref{lem:lower_bd_planar_ext_pf}: 
\begin{align*}
\ell_{\Omega_j}(\sigma)
  &\geq \int_a^1 k_{\Omega_j}(\sigma(s); \sigma^\prime(s)) ds \\
  &\geq a_\xi \int_a^1 \frac{\abs{\sigma^\prime(s)}}{\abs{\sigma(s)-x}} ds \\
  &\geq -a_\xi \int_a^1 \frac{d}{ds} \log \abs{\sigma(s)-x} ds \\
  & = a_{\xi}\big( \log \abs{\sigma(a)-x} -\log \abs{\sigma(1)-x} \big) = 2a_\xi(t_2-t_1).
\end{align*}  
Since the above estimate is independent of the choice of $\sigma$ with the above properties, it follows
from Result~\ref{res:integ_dist}-\eqref{item:dist_integ_absoc} that
\begin{align*}
K_{\Omega_j}( \sigma_{\xi, x}(t_1), \sigma_{\xi, x}(t_1) )\geq 2a_\xi(t_2-t_1).
\end{align*}
From this and the bound in the previous paragraph, the result follows.  
\end{proof}

\begin{lemma}\label{lem:not_infinite_bd_points}
Let $f:\Omega_1 \rightarrow \Omega_2$ be as in the statement of Theorem~\ref{thm:planar_domains}.
\begin{enumerate}
\item If $\xi \in \partial \Omega_1$, then $\lim_{z \rightarrow \xi} f(z)$ exists in $\overline{\Omega}_2^{End}$.
\item If $\xi \in \partial \Omega_2$, then $\lim_{z \rightarrow \xi} f^{-1}(z)$ exists in $\overline{\Omega}_1^{End}$.
\end{enumerate}
\end{lemma}
\begin{proof}
By symmetry it is enough to prove (1). By compactness, it is enough to fix a sequence $(z_n)_{n \geq 1}$ in $\Omega_1$ that converges to $\xi$ and show that $\lim_{n \rightarrow \infty} f(z_n)$ exists in $\overline{\Omega}_2^{End}$ and depends only on $\xi$. Considering the tail of $(z_n)_{n \geq 1}$, we can assume that for each $n$ there exists $x_n \in \partial \Omega_1 \cap \Oc_{\xi}$ where $z_n = \sigma_{\xi,x_n}(t_n)$ for some $t_n > 0$. Since $z_n \rightarrow \xi$, we must have $t_n \rightarrow \infty$ and $x_n \rightarrow \xi$. 

Observe that each $f \circ \sigma_{\xi,x_n}$ is a $(\lambda_\xi,0)$-almost-geodesic in $\Omega_2$. By construction, $f \circ \sigma_{\xi,x_n} \rightarrow f \circ \sigma_{\xi,\xi}$ locally uniformly. By Proposition~\ref{prop:planar}, $\Omega_2$ is locally a Goldilocks domain. Thus, by Theorem~\ref{thm:visible}, $\Omega_2$ is a visibility domain with respect to the Kobayashi distance. Thus, by Proposition~\ref{prop:limits_of_geodesic_rays}, the limits below exist and we have
$$ 
\lim_{n \rightarrow \infty} f(z_n) = \lim_{n \rightarrow \infty} f \circ \sigma_{\xi,x_n}(t_n)
= \lim_{t \rightarrow \infty} f \circ \sigma_{\xi,\xi}(t).
$$
\end{proof}

With these lemmas, we are now in a position to complete the proof of Theorem~\ref{thm:planar_domains}.

\begin{proof}[The proof of Theorem~\ref{thm:planar_domains}] The first step to proving Theorem~\ref{thm:planar_domains} is to produce a candidate for the homeomorphism whose existence the theorem claims.  In view of Lemma~\ref{lem:not_infinite_bd_points}, it suffices to establish the following
\medskip

\noindent \textbf{Claim I.} \emph{
\begin{enumerate}
\item If $\xi$ is an end of $\Omega_1$, then $\lim_{z \rightarrow \xi} f(z)$ exists in $\overline{\Omega}_2^{End}$.
\item If $\xi$ is an end of $\Omega_2$, then $\lim_{z \rightarrow \xi} f^{-1}(z)$ exists in $\overline{\Omega}_1^{End}$.
\end{enumerate}}
\smallskip

\noindent{By symmetry it is enough to prove part~(1) of the above claim.
Suppose, for a contradiction, that there exist sequences $(w_n)_{n \geq 1}$ and $(z_n)_{n \geq 1}$ in $\Omega_1$ such that $z_n \rightarrow \xi$, $w_n \rightarrow \xi$, but $f(z_n) \rightarrow \eta_1$, $f(w_n) \rightarrow \eta_2$, and $\eta_1 \neq \eta_2$. Since $\xi$ is an end of $\overline{\Omega}_1$, for each $n$ we can find a curve $\sigma_n : [0,1] \rightarrow \overline{\Omega}_1$ joining $z_n$ and $w_n$ such that $\sigma_n([0,1])$ leaves every compact subset of $\overline{\Omega}_1$. By perturbing each $\sigma_n$, we may assume that $\sigma_n([0,1]) \subset \Omega_1$.} 

Let $C_n := f( \sigma_n([0,1]))$. Passing to a subsequence and relabeling, if needed, we can assume that $C_n$ converges in the local Hausdorff topology to a closed set $C \subset \Cb^d$. Since $\sigma_n([0,1])$ leaves every compact subset of $\Omega_1$, we must have $C \subset \partial \Omega_2$. Since $\eta_1 \neq \eta_2$, $C$ must be non-empty. By construction, if $\eta \in C$, then there exist points $u_n \in \sigma_n([0,1])$ such that $f(u_n) \rightarrow \eta$. So, by Lemma~\ref{lem:not_infinite_bd_points},
$$
\lim_{z \rightarrow \eta} f^{-1}(z)=\lim_{n \rightarrow \infty} f^{-1} ( f(u_n)) = \lim_{n \rightarrow \infty} u_n = \xi.
$$
 Then, since $\xi$ is an end of $\overline{\Omega}_1$, 
\begin{align}\label{eqn:blow_up}
 \lim_{z \rightarrow \eta} \abs{f^{-1}(z)}=\infty. 
\end{align}

Arguing as in Lemma~\ref{lem:C_connected}, the closure of $C$ in $\overline{\Omega}^{End}$ is connected. Then $C$ intersected with each connected component of $\partial \Omega_2$ is connected. Since each connected component of $\partial \Omega_2$ is homeomorphic to $\Sb^1$ or $\Rb$, we can find a subset $I \subset C$ that is homeomorphic to an open interval. Since $\Omega_2$ is a Lipschitz domain, after possibly shrinking $I$ we can find a topological embedding $\tau : \overline{\Db} \rightarrow \Cb$ such that 
$$
\overline{I} \subset \tau(\overline{\Db}) \subset \Omega_2 \cup \overline{I}
$$
\begin{figure}[H]
  \centering
  \fbox{\includegraphics[width=0.95\textwidth]{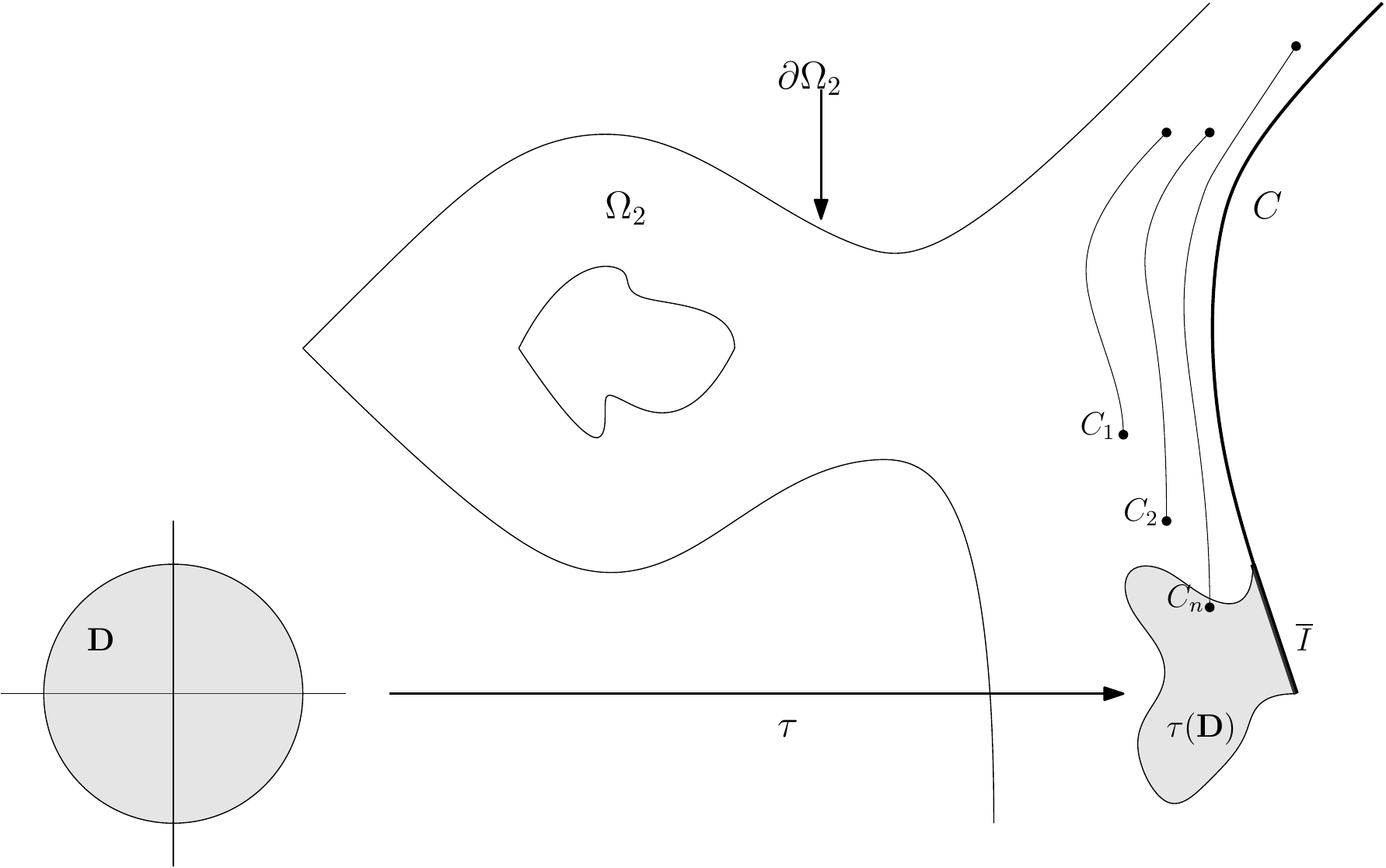}}
  \caption{The embedding of $\overline{\Db}$ by $\tau$}\label{fig:omega_2}
\end{figure}
\noindent{(see Figure~\ref{fig:omega_2}). Next, fix a Riemann mapping $\psi : \Db \rightarrow \tau(\Db)$. By the Carath\'eodory's extension theorem, $\psi$ extends to a homeomorphism $\overline{\Db} \cong \tau(\overline{\Db})$. In particular, $\psi^{-1}(I)$ is an open arc in $\Sb^1$.}

Since $\Omega_1$ is a Lipschitz domain, there exist $\epsilon > 0$ and $z_0 \in \Cb \setminus \Omega_1$ with $\Bb(z_0,\epsilon) \subset \Cb \setminus \Omega_1$. Then consider the map $\phi : \Db \rightarrow \Db \setminus\{0\}$ given by
$$
\phi(w) := \frac{\epsilon}{z_0-f^{-1}(\psi(w))}, \quad w \in \Db. 
$$
By \eqref{eqn:blow_up}, if $e^{i\theta} \in \psi^{-1}(I)$, then 
$$
\lim_{r \nearrow 1} \phi( re^{i\theta}) = 0.
$$
Then, since $\psi^{-1}(I)$ is an open arc in $\Sb^1$, the Luzin--Privalov Theorem (see~\cite[Chapter 2]{CL1966}) implies that $\phi \equiv 0$. Thus, we have a contradiction. This establishes Claim~I.

In view of Claim~I, we can now define the following maps. Define $F : \overline{\Omega}_1^{End} \rightarrow \overline{\Omega}_2^{End}$ by 
$$
F(z) = \begin{cases}
f(z), & \text{if } z \in \Omega_1, \\
\lim_{w \rightarrow z} f(w),  & \text{if } z \in \partial\overline{\Omega}_1^{End}.
\end{cases}
$$
Next, define $G : \overline{\Omega}_2^{End} \rightarrow \overline{\Omega}_1^{End}$ by 
$$
G(z) = \begin{cases}
f^{-1}(z), & \text{if } z \in \Omega_2, \\
\lim_{w \rightarrow z} f^{-1}(w),  & \text{if } z \in \partial\overline{\Omega}_2^{End}.
\end{cases}
$$

We would be done if we can prove:
\medskip

\noindent{\textbf{Claim II.} \emph{$F$ is a homeomorphism and $F^{-1}=G$.}}
\smallskip

\noindent{We first argue that $F$ and $G$ are continuous. By symmetry it is enough to show that $F$ is continuous. Since $\overline{\Omega}^{End}_2$ is compact, it is enough to fix a sequence $(\xi_n)_{n \geq 1}$ in $\overline{\Omega}^{End}_1$ for which $\xi_n \rightarrow \xi$, $F(\xi_n) \rightarrow \eta$, and then show that $\eta = F(\xi)$. By Claim~I, we may assume that $(\xi_n)_{n \geq 1}$ and $\xi$ are contained in $\partial \overline{\Omega}^{End}_1$. Since $\lim_{z \rightarrow \xi_n} f(z)=F(\xi_n)$, for each $n$ we can find $z_n \in \Omega_1$ sufficiently close to $\xi_n$ such that $z_n \rightarrow \xi$ and $f(z_n) \rightarrow \eta$. Then 
$$
\eta = \lim_{n \rightarrow \infty} f(z_n) = \lim_{z \rightarrow \xi} f(z)=F(\xi).
$$
So, $F$ is continuous.} 

Finally, notice that $(F\circ G)|_{\Omega_2} = \id_{\Omega_2}$ and $(G\circ F)|_{\Omega_1} = \id_{\Omega_1}$, whence, by density, $F \circ G = \id_{ \overline{\Omega}_2^{End}}$ and $G \circ F = \id_{ \overline{\Omega}_1^{End}}$. So, $F$ is a homeomorphism and $F^{-1}=G$.
\end{proof}

\section{Kobayashi geometry of planar domains: non-Gromov-hyperbolic domains}\label{sec:non-G-h}
In this section we construct two examples of planar domains $\Omega$ where $\Omega$ is a visiblity domain but $(\Omega, K_{\Omega})$ is not Gromov hyperbolic. In the first construction, the obstruction to Gromov hyperbolicity comes from the existence of a $\Zb^2$ action, while in the second the obstruction comes from asymptotic isometric embeddings of metric spaces admitting ``fat triangles.''

\subsection{Using $\boldsymbol{\mathbb{Z}^2}$ actions}
The first example comes directly from the following 

\begin{proposition} If
$$
\Omega := \Cb \setminus \cup_{m,n \in \Zb} \{ z \in \Cb : \abs{z -(m+ni)} < 1/4\},
$$
then $\Omega$ is a visibility domain, but $(\Omega, K_\Omega)$ is not Gromov hyperbolic.
\end{proposition} 

\begin{proof} Fix $z_0 \in \Omega$ and consider the map $F : \Zb^2 \rightarrow \Omega$ defined by 
$$
F(m,n) := z_0+ m+in.
$$
We claim that $F$ is a quasi-isometric embedding, which will imply that $(\Omega, K_\Omega)$ is not Gromov hyperbolic. Here, the metric on $\Zb^2$\,---\,which we shall denote by $d_{\Zb^2}$\,---\,is\linebreak
$d_{\Zb^2}((m_1, n_1), (m_2, n_2)) := |m_2 - m_1| + |n_2 - n_1|$.

Since $\Omega$ is invariant under translations by $\Zb + \Zb\bcdot\,i$, Lemma~\ref{lem:ext_cone} implies that there exists $C_1 > 0$ such that $k_\Omega(z;v) \geq C_1 \abs{v}$ for all $z \in \Omega$ and $v \in \Cb$ (even though it is a rather inefficient lower bound). Then 
$$
K_\Omega(z_1, z_2) \geq C_1 \abs{z_1-z_2}
$$
for all $z_1, z_2 \in \Omega$. Hence, for any $(m_1, n_1), (m_2, n_2)\in \Zb^2$, 
\begin{align*}
K_\Omega  (F(m_1,n_1), F(m_2,n_2)) & \geq C_1 \sqrt{ \abs{m_2-m_1}^2 + \abs{n_2-n_1}^2} \\
& \geq \frac{C_1}{\sqrt{2}} (\abs{m_2-m_1} +\abs{n_2-n_1}).
\end{align*}

Notice that the maps $S, T : \Omega \rightarrow \Omega$ given by $S(z) = z+1$ and $T(z) = z+i$ are commuting biholomorpisms. Further, $F(m,n) = S^m(T^n(z_0))$ and so
\begin{align*}
K_\Omega & (F(m_1,n_1), F(m_2,n_2))  = K_\Omega( S^{m_1-m_2}(z_0), T^{n_2-n_1}(z_0)) \\
& \leq K_\Omega( S^{m_1-m_2}(z_0), z_0) + K_\Omega( z_0, T^{n_2-n_1}(z_0)) \\
& = K_\Omega( S^{\abs{m_2-m_1}}(z_0), z_0) + K_\Omega( z_0, T^{\abs{n_2-n_1}}(z_0)) \\
& \leq \sum_{j=0}^{\abs{m_2-m_1}-1} K_\Omega( S^j (z_0), S^{j+1}(z_0))+\sum_{j=0}^{\abs{n_2-n_1}-1} K_\Omega( T^j (z_0), T^{j+1}(z_0)) \\
& = C_2 (\abs{m_2-m_1} +\abs{n_2-n_1})
\end{align*}
where $C_2 := \max\{ K_\Omega(T(z_0), z_0), K_\Omega(S(z_0), z_0)\}$. 

Now, note that the map $\Phi : (x,y)\mapsto (\lfloor{x}\rfloor, \lfloor{y}\rfloor)$ (where $\lfloor{\bcdot}\rfloor$ is the floor function) is a $(\sqrt{2}, 2)$-quasi-isometric embedding from $(\Rb^2, d_{{\rm Euc}})$ to $(\Zb^2, d_{\Zb^2})$. So, the last two chains of inequalities imply that $F\circ \Phi$ is a quasi-isometric embedding from $(\Rb^2, d_{{\rm Euc}})$ to $(\Omega, K_{\Omega})$. Secondly, note that as $\Omega$ is a planar hyperbolic domain, $(\Omega, K_{\Omega})$ is a geodesic metric space (recall that $K_{\Omega}$ coincides with the hyperbolic distance on $\Omega$). So, if $(\Omega, K_{\Omega})$ were Gromov hyperbolic, it would follow from \cite[Chapter~III.H, Theorem~1.9]{BH1999} that $(\Rb^2, d_{{\rm Euc}})$ is Gromov hyperbolic, which is clearly false. Thus, $(\Omega, K_\Omega)$ is not Gromov hyperbolic. However, it follows from Proposition~\ref{prop:planar} and Theorem~\ref{thm:visible} that $\Omega$ is a visibility domain.
\end{proof}

\subsection{Non-hyperbolic domains in the limit} 
For $s > 0$, define $A_s := \{ z \in \Cb : e^{-s} < \abs{z} < 1\}$. 

\begin{lemma}\label{lem:there_is_an_s_where_not_thin} For $s>0$, define
$$
\delta(s) := \inf\left\{ (x|y)_p^{A_s} - \min \left\{ (x|z)_p^{A_s}, (y|z)_p^{A_s}\right\} : x,y,z,p \in A_s \right\}.
$$
Then $\lim_{s \searrow 0} \delta(s) = -\infty$. 
\end{lemma}

\begin{proof}  Let $\Omega_s:= \{ x+iy \in\Cb : -s < x < 0, \ y \in \Rb\}$. Then, the map $\pi : \Omega_s \rightarrow A_s$, where $\pi(z) := e^z$, is a covering map. Hence, for any $z, w\in A_s$, 
$$
K_{A_s}(z,w) = \inf \left\{ K_{\Omega_s}(\wt{z}, \wt{w}) : \wt{z} \in \pi^{-1}\{z\}, \ \wt{w} \in \pi^{-1}\{w\}\right\}. 
$$ 
Since $\Omega_1\ni z \mapsto sz \in \Omega_s$ is an isometry, there exists $c > 0$ such that the curve $\wt{\sigma}_s(t) := -(s/2)+icst$  is a geodesic in $\Omega_s$ and hence $\sigma_s(t) := e^{-s/2}e^{icst}$ is a local geodesic in $A_s$. In fact, if $\abs{t_1-t_2} \leq \frac{\pi}{cs}$, then 
$$
K_{A_s}(\sigma(t_1), \sigma(t_2)) = \abs{t_1-t_2}.
$$

Now, \textbf{fix} $s>0$, and let $p=\sigma_s(0)$, $x=\sigma_s\left(\frac{\pi}{2cs}\right)$, $z =\sigma\left(-\frac{\pi}{cs}\right)$, and $y = \sigma\left(\frac{3\pi}{2cs}\right)$. Then, we compute to get 
$$
K_{A_s}(p,x) = K_{A_s}(x,z) = K_{A_s}(z,y) = K_{A_s}(y,p) = \frac{\pi}{2cs},
$$ 
and 
$$
K_{A_s}(x,y) = K_{A_s}(p,z) =  \frac{\pi}{cs}.
$$ 
Hence 
$$
(x|y)_p^{A_s} - \min \left\{ (x|z)_p^{A_s}, (y|z)_p^{A_s}\right\} = 0 - \min \left\{  \frac{\pi}{2cs}, \frac{\pi}{2cs} \right\} = - \frac{\pi}{2cs}. 
$$
Thus, $\delta(s) \leq  -\frac{\pi}{2cs}$, and the result follows.
\end{proof} 

We now describe the domain that will be the focus of this section. Fix a sequence $(s_n)_{n \geq 1}$ of numbers in $(0,1)$ such that for any $s \in (0,1)$ there exists a sequence $(n_j)_{j\geq 1}$ in $\Nb$ with $n_j \rightarrow \infty$ so that $s_{n_j} \nearrow s$. Next, for each $n \geq 1$, let $z_n := 3(n-1) \in \Zb_{\geq 0}$. Finally, 
\begin{itemize}
\item[$(\star)$] Let $\Omega \varsubsetneq \Cb$ be a connected domain with $\Cc^\infty$ boundary such that, for each $n \geq 1$,
\begin{align*}
z_n + A_{s_n} &\subset \Bb(z_n, 2) \cap \Omega \\
&\subset z_n + 
\left(A_{s_n} \cup \left\{ z \in \Cb : e^{-s_n} < \abs{z} < 2, \ -n^{-1} < \mathsf{Im}(z) < n^{-1} \right\} \right).
\end{align*}
\end{itemize}
Thus, $\Omega$ consists of a countable union of annuli centered at points along the real axis, each with outer radius $1$ and of varying inner radii, with thin strips, loosely speaking, joining adjacent annuli. To elaborate further, consider the auxiliary domain
$$
\wt{\Omega} := \bigcup_{n=1}^\infty \left(z_n + 
\left(A_{s_n} \cup \left\{ z \in \Cb : e^{-s_n} < \abs{z} < 2, \ -n^{-1} < \mathsf{Im}(z) < n^{-1} \right\} \right)
\right).
$$
Our $\Omega$ looks approximately like $\wt{\Omega}$: more precisely, $\Omega \varsubsetneq \wt{\Omega}$ is obtained by smoothening the boundary of $\wt{\Omega}$ but ensuring that $(z_n + A_{s_n}) \subset \Omega$ for each $n\geq 1$ and such that the second inclusion in the description $(\star)$ holds true for each $n\geq 1$.

\begin{lemma}\label{lem:convergence_on_each_ann} Let $\Omega$ be as described above. If $s \in (0,1)$, $s_{n_j} \nearrow s$, and $z,w \in A_s$, then 
$$
K_{A_s}(z,w) = \lim_{j \rightarrow \infty} K_\Omega(z+z_{n_j}, w+z_{n_j}). 
$$
\end{lemma} 
\begin{proof} \textbf{Fix} $s \in (0,1)$, $s_{n_j} \nearrow s$,  and $z, w \in A_s$. Notice that 
\begin{align}
\limsup_{j \rightarrow \infty} K_\Omega(z+z_{n_j}, w+z_{n_j}) & \leq  \limsup_{j \rightarrow \infty} K_{z_{n_j}+A_{s_{n_j}}}(z+z_{n_j}, w+z_{n_j}) \notag \\
& =\limsup_{j \rightarrow \infty} K_{A_{s_{n_j}}}(z,w) \leq K_{A_s}(z,w) \; \; \forall z,w \in A_s
\label{eqn:upper_bound_dumb_limit}
\end{align}
since $\big(A_{s_{n_j}}\big)_{\!j\geq 1}$ is a nested sequence of annuli. 

Let $\varphi_j : \Db \rightarrow \Omega$ be a holomorphic covering map with $\varphi_j(0) = z+z_{n_j}$ and $\varphi_j(t_j) = w+z_{n_j}$ where $t_j \in (0,1)$ and $K_{\Db}(0,t_j) = K_\Omega(z+z_{n_j},w+z_{n_j})$. Notice that $\{t_j: j\geq 1\}$ is relatively compact in $[0,1)$ by \eqref{eqn:upper_bound_dumb_limit}.

Consider $\Phi_j : \Db \rightarrow \Cb$ given by $\Phi_j(\zeta) = -z_{n_j}+\varphi_j(\zeta)$. Then, $\Phi_j(0) = z$ and $\Phi_j(\Db) \subset \Omega \setminus \Bb(0,e^{-2s})$ when $j$ is sufficiently large (recall that $z\in A_s$). Then, by Montel's theorem, we can find a subsequence $(\Phi_{j_k})_{k \geq 1}$ such that 
\begin{equation}\label{eqn:interim}
\lim_{k \rightarrow \infty} K_{\Omega}\left(z+z_{n_{j_k}},w+z_{n_{j_k}}\right)
= \liminf_{j \rightarrow \infty} K_\Omega(z+z_{n_j}, w+z_{n_j}) 
\end{equation}
and $\Phi_{j_k}$ converges locally uniformly to a holomorphic map $\Phi : \Db \rightarrow \Cb$. Then 
$$
\Phi(\Db) \cap \Bb(0,2) \subset \overline{A_s} \cup \{ z \in \Cb : e^{-s} \leq \abs{z} \leq 2 \text{ and } \mathsf{Im}(z) = 0\}.
$$
Since $\Phi(\Db)$ is open, this implies that $\Phi(\Db) \cap \Bb(0,2) \subset A_s$. Since $\Phi(\Db)$ is connected, then $\Phi(\Db) \subset A_s$. Then, by \eqref{eqn:interim} and the definition of $t_j$, 
\[
\liminf_{j \rightarrow \infty} K_\Omega(z+z_{n_j}, w+z_{n_j}) 
=  \lim_{k \rightarrow \infty}K_{\Db}(0,t_{j_k}) \geq  \lim_{k \rightarrow \infty} K_{A_s}(\Phi(0),\Phi(t_{j_k})) = K_{A_s}(z,w). 
\]
From this and \eqref{eqn:upper_bound_dumb_limit}, the result follows.
\end{proof}

We now have all the ingredients for the second type of example alluded to above.

\begin{proposition}
Let $\Omega \subset \Cb$ be as described by $(\star)$. Then, $\Omega$ is a visibility domain, but $(\Omega, K_\Omega)$ is not Gromov hyperbolic.
\end{proposition} 

\begin{proof} Fix $m \in \Nb$. By Lemma~\ref{lem:there_is_an_s_where_not_thin} there exist $s >0$ and $p,x,y,z \in A_s$ such that 
$$
(x|y)_{p}^{A_s} \leq \min \left\{ (x|z)_p^{A_s}, (y|z)_p^{A_s} \right\} - (m+1).
$$
Fix $n_j\rightarrow \infty$ such that $s_{n_j} \nearrow s$. Then by Lemma~\ref{lem:convergence_on_each_ann}
$$
(x+z_{n_j}|y+z_{n_j})_{p+z_{n_j}}^{\Omega} \leq \min \left\{ (x+z_{n_j}|z+z_{n_j})_{p+z_{n_j}}^{\Omega}, (y+z_{n_j}|z+z_{n_j})_{p+z_{n_j}}^{\Omega} \right\} - m
$$
for $j$ sufficiently large. 

Since $m \in \Nb$ was arbitrary, $(\Omega, K_\Omega)$ is not Gromov hyperbolic. However, as $\Omega$ has smooth boundary, by Proposition~\ref{prop:planar} and Theorem~\ref{thm:visible}, $\Omega$ is a visibility domain.
\end{proof}

\section{The relationship between Gromov hyperbolicity and the weak visibility property}
While this section is devoted to the proof of Theorem~\ref{thm:Gromov_visi}, we shall take a detour to present an alternative formulation for $(\Omega, K_{\Omega})$ to be Gromov hyperbolic. This result provides a clean and intuitive proof of Theorem~\ref{thm:Gromov_visi}. Such a result may also be of independent interest. Thus, we begin with:

\subsection{A slim-triangles formulation of Gromov hyperbolicity of $\boldsymbol{(\Omega, K_{\Omega})}$}\label{ssec:slim}
We begin with a pair of definitions inspired by our notion of $(1,\kappa)$-almost-geodesics.

\begin{definition}
Let $\Omega \subset \Cb^d$ be a Kobayashi hyperbolic domain and $I \subset \Rb$ be an interval. For $\kappa \geq 0$, a map $\sigma:I \to \Omega$ is called a \emph{$(1, \kappa)$-near-geodesic} if $\sigma$ is continuous and for all $s,t \in I$,  
$$
|t-s| - \kappa \leq K_{\Omega}(\sigma(s), \sigma(t)) \leq |t-s| +  \kappa.
$$
\end{definition}

\begin{definition}\label{def:admissible}
Let $\Omega \subset \Cb^d$ be a Kobayashi hyperbolic domain. Fix $\kappa \geq 0$. We shall call a family $\mathscr{F}$ of
$(1, \kappa)$-near-geodesics a \emph{$\kappa$-admissible family} if the following conditions are satisfied:
\begin{itemize}
\item for each pair of distinct points $z, w\in \Omega$, there exists a $(1, \kappa)$-near-geodesic $\sigma \in \mathscr{F}$, $\sigma : [0, T] \rightarrow \Omega$, such that $\sigma(0) = z$ and $\sigma(T) = w$;

\item if a $(1, \kappa)$-near-geodesic $\sigma : [0, T] \rightarrow \Omega$ is in $\mathscr{F}$, then so is $\wt{\sigma}$, where
$$
\wt{\sigma}(t) := \sigma(T-t), \quad t\in [0, T].
$$
\end{itemize}
\end{definition}

For brevity, if $z, w\in \Omega$ (not necessarily distinct), we shall use the notation $\nrg{\sigma}{z}{w}$ to denote that $\sigma : [0, T] \rightarrow \Omega$ is a $(1, \kappa)$-near-geodesic, for some specified $\kappa \geq 0$, joining $z$ and $w$. Whether [$\sigma(0) = z$ and $\sigma(T) = w$] or [$\sigma(T) = w$ and $\sigma(0) = z$] will be clear from the context. The notion introduced in Definition~\ref{def:admissible} is non-vacuous because, in view of Proposition~\ref{prop:almost_geod_exist}, given a $\kappa > 0$, the family of all $(1, \kappa)$-almost-geodesics is a $\kappa$-admissible family. The above discussion sets the stage for the principal collection of concepts of this section.

\begin{definition}\label{defn:triangles}
Let $\Omega \subset \Cb^d$ be a Kobayashi hyperbolic domain. Fix $\kappa \geq 0$.
\begin{enumerate}
\item\label{item:kappa_t} A \emph{$\kappa$-triangle} is a triple of $(1, \kappa)$-near-geodesics $\nrg{\alpha}{b}{c}$, $\nrg{\beta}{a}{c}$ and $\nrg{\gamma}{a}{b}$, where $a, b, c\in \Omega$. We call $a$, $b$ and $c$ the \emph{vertices}, and the \textbf{images} of $\alpha$, $\beta$ and $\gamma$\,---\,which, by a mild abuse of notation, we shall also denote simply as $\alpha$, $\beta$ and $\gamma$, respectively\,---\,the \emph{sides} of this $\kappa$-triangle.

\item Let $\delta \geq 0$. A $\kappa$-triangle $\triangle$ is said to be \emph{$\delta$-slim} if, for each side $\sigma$ of $\triangle$,
$$
  \sigma \subset \bigcup\nolimits_{\Sigma \in \triangle\setminus \{\sigma\}} \{z \in \Omega : K_{\Omega}(z, \Sigma) \leq \delta\}
$$

\item Let $\mathscr{F}$ be a $\kappa$-admissible family of $(1, \kappa)$-near-geodesics. We say that $\Omega$ satisfies the \emph{$(\mathscr{F}, \kappa, \delta)$-Rips condition}, where $\delta \geq 0$, if every $\kappa$-triangle in $\Omega$ whose sides are (the images of) $(1, \kappa)$-near-geodesics belonging to $\mathscr{F}$ is $\delta$-slim.
\end{enumerate}
\end{definition}

We first need a couple of lemmas.

\begin{lemma}\label{lem:Gp_estimate}
Let $\Omega \subset \Cb^d$ be a Kobayashi hyperbolic domain. Fix $\kappa \geq 0$ and fix a $\kappa$-admissible family $\mathscr{F}$ of $(1,\kappa)$-near-geodesics. For any $\sigma : [0, T] \rightarrow \Omega$ in $\mathscr{F}$ and any $z \in \Omega$, we have $\Gp{\sigma(0)}{\sigma(T)}{z} \leq K_{\Omega}(z, \sigma) + 3\kappa/2$.
\end{lemma}
\begin{proof}
Fix $\sigma \in \mathscr{F}$ and $z \in \Omega$. Let $\tau \in [0, T]$ be such that $K_{\Omega}(z, \sigma) = K_{\Omega}(z, \sigma(\tau))$. We estimate:
\begin{align*}
K_{\Omega}(\sigma(0), z) + K_{\Omega}(z, \sigma(T))
&\leq K_{\Omega}(\sigma(0), \sigma(\tau)) + K_{\Omega}(\sigma(\tau), \sigma(T)) + 2K_{\Omega}(z, \sigma) \\
&\leq \big(|\tau - 0| + |T - \tau| + 2\kappa \big) + 2K_{\Omega}(z, \sigma) \\
&\leq K_{\Omega}(\sigma(0), \sigma(T)) + 3\kappa + 2K_{\Omega}(z, \sigma).
\end{align*}
The second and third inequalities above are due to the fact that $\sigma$ is a $(1, \kappa)$-near-geodesic.
Thus
$$
\frac{K_{\Omega}(\sigma(0), z) + K_{\Omega}(z, \sigma(T))-K_{\Omega}(\sigma(0), \sigma(T))}{2}
\leq \frac{3\kappa}{2} + K_{\Omega}(z, \sigma).
$$
As $\sigma$ and $z$ were picked arbitrarily, the result follows.
\end{proof}
  
\begin{lemma}\label{lem:technical}
Let $\Omega \subset \Cb^d$, $\kappa \geq 0$ and $\mathscr{F}$ be as in Lemma~\ref{lem:Gp_estimate}. Suppose
\begin{itemize}
\item[$(\blz)$] for any $\sigma : [0, T] \rightarrow \Omega$ in $\mathscr{F}$, any point $p\in \sigma([0, T])$ and any $z \in \Omega$, we have
$$
\min\left\{ \Gp{\sigma(0)}{z}{p}, \Gp{z}{\sigma(T)}{p} \right\} \leq \delta.
$$
\end{itemize}
Then, we have
$$
K_{\Omega}(z, \sigma) \leq \Gp{\sigma(0)}{\sigma(T)}{z} + 2\delta + 3\kappa/2
$$
for every $\sigma : [0, T] \rightarrow \Omega$ in $\mathscr{F}$ and every $z \in \Omega$
\end{lemma}
\begin{proof}
Fix $\sigma \in \mathscr{F}$ and $z \in \Omega$. By $(\blz)$, $[0,T]$ is the union
$\{t : \Gp{\sigma(0)}{z}{\sigma(t)}\leq \delta\}\cup \{t : \Gp{z}{\sigma(T)}{\sigma(t)}\leq \delta\}$
of non-empty closed sets. So, there is a $\tau \in [0, T]$ such that  
$$
K_{\Omega}(z, \sigma(\tau)) + K_{\Omega}(q, \sigma(\tau)) - K_{\Omega}(z, q) \leq 2\delta
$$
for $q = \sigma(0), \sigma(T)$. Thus:
\begin{align*}
2K_{\Omega}(z, \sigma) &\leq 2K_{\Omega}(z, \sigma(\tau)) \\
&\leq 4\delta + K_{\Omega}(z, \sigma(0)) + K_{\Omega}(z, \sigma(T))
	- \big( K_{\Omega}(\sigma(0), \sigma(\tau)) + K_{\Omega}(\sigma(\tau), \sigma(T)) \big) \\
&\leq 4\delta + K_{\Omega}(z, \sigma(0)) + K_{\Omega}(z, \sigma(T)) 
	- \big( |\tau - 0| + |T - \tau| - 2\kappa \big) \\
&\leq 4\delta + K_{\Omega}(z, \sigma(0)) + K_{\Omega}(z, \sigma(T)) - K_{\Omega}(\sigma(0), \sigma(T)) + 3\kappa \\
&= 4\delta + 2\Gp{\sigma(0)}{\sigma(T)}{z} + 3\kappa.
\end{align*}
The third and fourth inequalities above are due to the fact that $\sigma$ is a $(1, \kappa)$-near-geodesic. As $\sigma \in \mathscr{F}$ and $z \in \Omega$ were arbitrarily chosen, the result follows.
\end{proof}

Lemma~\ref{lem:technical} assists in the key result of this section. This result is an equivalent formulation of Gromov hyperbolicity of $(\Omega, K_{\Omega})$, $\Omega$ being Kobayashi hyperbolic, in terms of a slim-triangles criterion\,---\,resembling what is known for proper Gromov hyperbolic spaces\,---\,\textbf{without} any assumption on properness of $(\Omega, K_{\Omega})$. If $(\Omega, K_{\Omega})$ is Gromov hyperbolic, then we say that $(\Omega, K_{\Omega})$ is \emph{$\delta$-hyperbolic} to refer to the constant $\delta \geq 0$ that appears in the inequality \eqref{eq:Gromov_ineq} defining Gromov hyperbolicity. 

\begin{theorem}\label{thm:equiv_GH}
Let $\Omega \subset \Cb^d$ be a Kobayashi hyperbolic domain. Let $\kappa \geq 0$ be such that there exists some $\kappa$-admissible family of $(1, \kappa)$-near-geodesics. Then:
\begin{enumerate}
\item\label{item:infer_slim} Suppose $(\Omega, K_{\Omega})$ is $\delta$-hyperbolic for some $\delta \geq 0$. Then, for any $\kappa$-admissible family $\mathscr{F}$, $\Omega$ satisfies the $(\mathscr{F}, \kappa, 3\delta+6\kappa)$-Rips condition.

\item\label{item:infer_hyp} Suppose, for some $\kappa$-admissible family $\mathscr{F}$, $\Omega$ satisfies the $(\mathscr{F}, \kappa, \delta)$-Rips condition. Then $(\Omega, K_{\Omega})$ is $(3\delta + 6\kappa)$-hyperbolic.
\end{enumerate}
\end{theorem}
\begin{proof}
We first prove (\ref{item:infer_slim}). Thus, assume that $(\Omega, K_{\Omega})$ is $\delta$-hyperbolic for some $\delta \geq 0$. Next, \textbf{fix} a $\kappa$-admissible family $\mathscr{F}$. Consider an arbitrary $\sigma :  [0, T] \rightarrow \Omega$ in $\mathscr{F}$, an arbitrary point $p\in \sigma([0, T])$ and an arbitrary point $z \in \Omega$. Let $\tau \in [0, T]$ be such that $\sigma(\tau) = p$. By $\delta$-hyperbolicity,
$$
\min\left\{ \Gp{\sigma(0)}{z}{p}, \Gp{z}{\sigma(T)}{p} \right\} 
\leq \Gp{\sigma(0)}{\sigma(T)}{p} + \delta
\leq 3\kappa/2 + \delta, 
$$
where the second inequality follows from Lemma~\ref{lem:Gp_estimate} by taking $z=p$ therein. We have just established that
\begin{equation}\label{eq:diamond_1}
\text{the condition $(\blz)$ holds with the parameter $\delta$ replaced by $(\delta + 3\kappa/2)$.}
\end{equation}

Now, consider a $\kappa$-triangle $\triangle$ whose vertices are $a$, $b$ and $c$; whose sides $\alpha$, $\beta$ and $\gamma$ belong to $\mathscr{F}$; and where the labels for these sides are as described in Definition~\ref{defn:triangles}-(\ref{item:kappa_t}). It suffices to show that $K_{\Omega}(p, \beta\cup \gamma) \leq (3\delta + 6\kappa)$ for an arbitrary $p$ belonging to the image of $\alpha$. In view of \eqref{eq:diamond_1} and Lemma~\ref{lem:technical}, we have
\begin{align*}
K_{\Omega}(p, \beta) &\leq \Gp{a}{c}{p} + 2\delta + 9\kappa/2 \; \text{ and} \\
K_{\Omega}(p, \gamma) &\leq \Gp{a}{b}{p} + 2\delta + 9\kappa/2,
\end{align*}
whence, by $\delta$-hyperbolicity, it follows from the above that
$$
K_{\Omega}(p, \beta\cup \gamma) \leq \Gp{b}{c}{p} + 3\delta + 9\kappa/2.
$$
Recall: $p$ is in the image of the $(1,\kappa)$-near-geodesic $\nrg{\alpha}{b}{c}$. Thus, by Lemma~\ref{lem:Gp_estimate}, $\Gp{b}{c}{p} \leq 3\kappa/2$. So, from the last inequality, we have
$$
K_{\Omega}(p, \beta\cup \gamma) \leq 3\delta + 6\kappa.
$$
As remarked above, this suffices to show that $\triangle$ is $(3\delta + 6\kappa)$-slim. It follows that $\Omega$ satisfies the $(\mathscr{F}, \kappa, 3\delta+6\kappa)$-Rips condition.

To prove (\ref{item:infer_hyp}), suppose $\mathscr{F}$ is a $\kappa$-admissible family such that $\Omega$ satisfies the $(\mathscr{F}, \kappa, \delta)$-Rips condition. Consider an arbitrary $\sigma :  [0, T] \rightarrow \Omega$ in $\mathscr{F}$, an arbitrary point $p\in \sigma([0, T])$ and an arbitrary point $z \in \Omega$. By $\kappa$-admissibility, there exist $(1,\kappa)$-near-geodesics
$$
\nrg{\rho}{z}{\sigma(0)} \quad\text{and}
\quad \nrg{\upsilon}{z}{\sigma(T)},
$$
whence $\rho, \sigma$ and $\upsilon$ form the sides of a (perhaps degenerate) $\kappa$-triangle. As $p$ lies in $\sigma$, by hypothesis there is a point $q \in \rho \cup \upsilon$ such that $K_{\Omega}(p, q) \leq \delta$. Assume that $q \in \rho$. Suppose $\rho : [0, R] \to \Omega$ is such that  $\rho(0) = z$, and let $\tau \in [0, R]$ be such that $\rho(\tau) = q$. We compute, using a by-now-familiar move for the second inequality: 
\begin{align*}
K_{\Omega}(\sigma(0), p) + K_{\Omega}(z, p)
	&\leq K_{\Omega}(\sigma(0), q) + K_{\Omega}(q, p) + K_{\Omega}(p, q) + K_{\Omega}(q, z) \\
	&\leq 2\delta + K_{\Omega}(\sigma(0), z) + 3\kappa \\
\Rightarrow\,\Gp{\sigma(0)}{z}{p} &\leq \delta + 3\kappa/2.
\end{align*}
The above inequality was obtained under the assumption that $q \in \rho$. If, instead, $q\in \upsilon$ then we would get the inequality $\Gp{\sigma(T)}{z}{p} \leq \delta + 3\kappa/2$. Since at least one of the last two inequalities must hold, we get
$$
\min\left\{ \Gp{\sigma(0)}{z}{p}, \Gp{z}{\sigma(T)}{p} \right\} \leq \delta + 3\kappa/2,
$$
which leads to the conclusion:
\begin{equation}\label{eq:diamond_2}
\text{the condition $(\blz)$ holds with the parameter $\delta$ replaced by $(\delta + 3\kappa/2)$.}
\end{equation}

Now, consider points $a, b, c, o\in \Omega$. Form a $\kappa$-triangle $\triangle$ whose vertices are $a$, $b$ and $c$; whose sides $\alpha$, $\beta$ and $\gamma$ belong to $\mathscr{F}$, and where the labels for these sides are as described in Definition~\ref{defn:triangles}-(\ref{item:kappa_t}). By Lemma~\ref{lem:Gp_estimate}, we have
$$
\Gp{a}{b}{o} \leq K_{\Omega}(o, \gamma) + 3\kappa/2 \quad\text{and}
\quad \Gp{b}{c}{o} \leq K_{\Omega}(o, \alpha) +3\kappa/2.
$$
Owing to the $(\mathscr{F}, \kappa, \delta)$-Rips condition, there exist points $p_{\alpha}$ and $p_{\gamma}$ that lie in the sides $\alpha$ and $\gamma$, respectively, of $\triangle$ such that
$$
K_{\Omega}(o, p_{\alpha}) \leq  K(o, \beta) + \delta \quad\text{and}
\quad K_{\Omega}(o, p_{\gamma}) \leq  K(o, \beta) + \delta.
$$
By the last four inequalities
$$
\min\left\{ \Gp{a}{b}{o}, \Gp{b}{c}{o} \right\} \leq K(o, \beta) + \delta + 3\kappa/2.
$$
This implies\,---\,in view of \eqref{eq:diamond_2} and Lemma~\ref{lem:technical}\,---\,that
$$
\min\left\{ \Gp{a}{b}{o}, \Gp{b}{c}{o} \right\} \leq \Gp{a}{c}{o} + 2\big(\delta + 3\kappa/2 \big) + \delta + 3\kappa.
$$
As $a, b, c$ and $o$ were arbitrary, it follows that $(\Omega, K_{\Omega})$ is $(3\delta + 6\kappa)$-hyperbolic. 
\end{proof}

\subsection{The proof of Theorem~\ref{thm:Gromov_visi}}\label{ssec:Gromov_visi}
In this proof, $o \in \Omega$ is some point which will stay fixed. Furthermore:
\begin{itemize}
\item the notation ``$\glim p$'' will denote approach to $p \in \partial_{G}\Omega$ with respect to
the topology on $(\Omega \cup \partial_{G}\Omega)$,
\item the notation ``$\rightarrow \xi$'' will denote approach to $\xi \in \partial\overline{\Omega}^{End}$  with respect to the topology on $\overline{\Omega}^{End}$,
\item $\lim$ will denote limits in all contexts \emph{other than} the topology on $(\Omega \cup \partial_{G}\Omega)$. 
\end{itemize}

Let $\delta\geq 0$ be such that $(\Omega, K_{\Omega})$ is $\delta$-hyperbolic. Let $\wt{{\sf id}}_{\Omega}$ denote the homeomorphism from $(\Omega \cup \partial_{G}\Omega)$ onto $\overline{\Omega}^{End}$ which extends ${\sf id}_{\Omega}$.
\medskip

\noindent{{\bf Claim~I.} \emph{$(\Omega, K_{\Omega})$ is Cauchy-complete}.}

\noindent{Let $(z_n)_{n\geq 1}\subset \Omega$ be a sequence having no limit points in $\Omega$. Since $\overline{\Omega}^{End}$ is compact, we may assume without loss of generality that $z_n \rightarrow \xi$ for some $\xi \in \partial\overline{\Omega}^{End}$. Consider any subsequence $(z_{n_k})_{k\geq 1}$ such that $\big(K_{\Omega}(o, z_{n_k})\big)_{k\geq 1}$ is convergent in $[0, +\infty]$. By hypothesis, $(\Omega \cup \partial_{G}\Omega)$ is compact. So, we may assume, by passing to a subsequence and relabeling if needed, that there exists a point $p \in \partial_{G}\Omega$ such that
$$
z_{n_k}\!\!\glim p \quad \text{and}
\quad \wt{{\sf id}}_{\Omega}(p) = \xi.
$$
By Result~\ref{res:Gromov_comptn}-\eqref{item:convg}, $(z_{n_k})_{k\geq 1}$ is a Gromov sequence. So:
$$
+\infty = \liminf_{j, k\to \infty}\Gp{z_{n_j}}{z_{n_k}}{o} 
\leq \liminf_{k\to \infty}\Gp{z_{n_k}}{z_{n_k}}{o} =  \liminf_{k\to \infty}K_{\Omega}(o, z_{n_k}).
$$
Hence, the function $K_{\Omega}(o, \bcdot)$ is proper. By Proposition~\ref{prop:Hopf_Rinow}, $(\Omega, K_{\Omega})$ is Cauchy-complete.}
\medskip
 
\noindent{{\bf Claim~II.} \emph{$\Omega$ is a weak visibility domain.}}

\noindent{Assume, aiming for a contradiction, that $\Omega$ is not a weak visibility domain. Then, there exist a constant $\kappa \geq 0$, points $\xi, \eta \in \partial\overline{\Omega}^{End}$, $\overline{\Omega}^{End}$-open neighborhoods $V_{\xi}$ and $V_{\eta}$ of $\xi$ and $\eta$, respectively, such that $\clos{V_{\xi}} \cap \clos{V_{\eta}} = \emptyset$ and a sequence of $(1, \kappa)$-almost-geodesics $(\sigma_n)_{n\geq1}$, $\sigma_n: [0, T_n] \rightarrow \Omega$ such that
$$
\sigma_n(0) \in V_{\xi}, \ \sigma(T_n) \in V_{\eta} \; \forall n \geq 1, \quad\text{and}
\quad \lim_{n\to \infty}K_{\Omega}(o, \sigma_n) = +\infty.
$$
In the last assertion, we have used the fact that $K_{\Omega}(o, \bcdot)$ is proper. Passing to a subsequence and relabeling if needed, we may assume:
\begin{itemize}
\item $\sigma_n(0) \rightarrow \xi^\prime$ and $\sigma_n(T_n) \rightarrow \eta^\prime$, where $\xi^\prime, \eta^\prime \in \partial\overline{\Omega}^{End}$;
\item $\big(K_{\Omega}(o, \sigma_n(0))\big)_{n\geq 1}$ and $\big(K_{\Omega}(o, \sigma_n(T_n))\big)_{n\geq 1}$ are monotone increasing sequences.
\end{itemize}
Since $\clos{V_{\xi}} \cap \clos{V_{\eta}} = \emptyset$, $\xi^\prime \neq \eta^\prime$. Write:
$$
z_n := \sigma_n(0) \quad\text{and}
\quad w_n := \sigma_n(T_n).
$$
By our hypothesis (note that $z_n \rightarrow \xi^\prime$ and $w_n \rightarrow \eta^\prime$ in $\overline{\Omega}^{End}$) and by Result~\ref{res:Gromov_comptn}-\eqref{item:convg}, we have
$$
[(z_n)_{n \geq 1}] = \wt{{\sf id}}_{\Omega}^{-1}(\xi^\prime) \quad\text{and}
\quad  [(w_n)_{n \geq 1}] = \wt{{\sf id}}_{\Omega}^{-1}(\eta^\prime).
$$
Thus, $[(z_n)_{n \geq 1}] \neq [(w_n)_{n \geq 1}]$ and so, by definition
$$
\limsup_{m, n\to \infty}\Gp{z_m}{w_n}{o} < +\infty.
$$
In particular, therefore
\begin{equation}\label{eq:fin_prod_lim}
\limsup_{n\to \infty}\Gp{z_n}{w_n}{o} < +\infty.
\end{equation}}

Now, fix (we make use of Claim~I once again, together with the Hopf--Rinow theorem: Result~\ref{res:H-R})
\begin{itemize}
\item some geodesic of $(\Omega, K_{\Omega})$ joining $o$ and $z_n$ and call it $\rho_n$;
\item some geodesic of $(\Omega, K_{\Omega})$ joining $o$ and $w_n$ and call it $\upsilon_n$.
\end{itemize}
In the notation of Theorem~\ref{thm:equiv_GH}, and with $\kappa$ as above, let $\mathscr{F}$ be the family of all $(1, \kappa)$-near-geodesics in $\Omega$. Clearly, $\mathscr{F}$ is $\kappa$-admissible. Finally, let $\triangle_n$ denote the $\kappa$-triangle whose sides are $\rho_n$, $\sigma_n$ and $\upsilon_n$. Let us write $\delta^* := (3\delta + 6\kappa)$. By part~\eqref{item:infer_slim} of Theorem~\ref{thm:equiv_GH} we get that $\Omega$ satisfies the $(\mathscr{F}, \kappa, \delta^*)$-Rips condition. In particular,
$$
  \sigma_n \subset \bigcup\nolimits_{\Sigma \in \triangle_n\setminus \{\sigma_n\}}
  \{z \in \Omega : K_{\Omega}(z, \Sigma) \leq \delta^*\}
$$
for each $n \geq 1$. Since $\sigma_n([0, T_n])$ is connected, we can find a point
$$
b_n \in \{z \in \Omega : K_{\Omega}(z, \rho_n) \leq \delta^*\} \cap 
\{z \in \Omega : K_{\Omega}(z, \upsilon_n) \leq \delta^*\}
$$
for each $n \geq 1$. Next, fix $a_n \in \rho_n$ and $c_n \in \upsilon_n$ such that
$$
K_{\Omega}(a_n, b_n), K_{\Omega}(c_n, b_n) \leq \delta^*
$$
for each $n \geq 1$. Clearly,
$$
K_{\Omega}(o, a_n) \geq K_{\Omega}(o, \sigma_n) - \delta^*, \quad
K_{\Omega}(o, c_n) \geq K_{\Omega}(o, \sigma_n) - \delta^*,
$$
for each $n \geq 1$. From the above and the fact that $\rho_n$ and $\upsilon_n$ are geodesics,
\begin{align*}
K_{\Omega}(a_n, z_n) &\leq K_{\Omega}(o, z_n) - K_{\Omega}(o, \sigma_n) + \delta^*, \\
K_{\Omega}(c_n, w_n) &\leq K_{\Omega}(o, w_n) - K_{\Omega}(o, \sigma_n) + \delta^*,
\end{align*}
for each $n \geq 1$. By the triangle inequality and the last two inequalities, we get:
\begin{align*}
K_{\Omega}(z_n, b_n) &\leq K_{\Omega}(o, z_n) - K_{\Omega}(o, \sigma_n) + 2\delta^*, \\
K_{\Omega}(w_n, b_n) &\leq K_{\Omega}(o, w_n) - K_{\Omega}(o, \sigma_n) + 2\delta^*,
\end{align*}
for each $n \geq 1$. From the last two inequalities, it follows that
$$
2K_{\Omega}(o, \sigma_n) \leq 2\Gp{z_n}{w_n}{o} + 4\delta^*
$$
for each $n \geq 1$. From this and \eqref{eq:fin_prod_lim} we have
$$
\limsup_{n\to \infty}K_{\Omega}(o, \sigma_n) \leq \limsup_{n\to \infty}\Gp{z_n}{w_n}{o} + 2\delta^*< +\infty,
$$
which contradicts the fact that $\lim_{n\to \infty}K_{\Omega}(o, \sigma_n) = +\infty$. Hence, our assumption must be false, which establishes the Claim~II. \hfill $\qed$

\subsection{An application of Theorem~\ref{thm:Gromov_visi}}
\begin{proof}[The proof of Theorem~\ref{thm:Gromov_visi_bdy}]
Suppose the conclusion of the theorem is false. Then there exists a germ of a complex $1$-dimensional variety, say $V$, in $\partial\Omega$. Let $p$ be a regular point of $V$. By the condition on $\partial\Omega$, it is easy to see that there exists a neighborhood $\Uc$ of $p$, a unit vector $\nu$, and a constant $\epsilon > 0$ such that
\begin{itemize}
  \item $\Uc \cap V$ is the image of an injective holomorphic map $\varphi: \Db \to \Cb^d$,
  \item $(\Uc \cap V)+t\nu \subset \Omega$ for every $t \in (0,\epsilon)$.
\end{itemize}
Write $\xi := \varphi(1/2)$ and $\eta := \varphi(-1/2)$; by injectivity of $\varphi$, $\xi\neq \eta$. Next, write
$$
  \varphi_{n} := \varphi + \left(\tfrac{\epsilon}{n+1}\right)\nu,
$$
and set $z_n := \varphi_n(1/2)$, $w_n := \varphi_n(-1/2)$ for each $n\in \Nb$. Finally, let $\gamma : ([0, T], 0, T)\to (\Db, -1/2, 1/2)$ denote the geodesic with respect to the Poincar{\'e} distance on $\Db$ from $-1/2$ to $1/2$ which lies in $(-1, 1)$, and define the path $\sigma_n := \varphi_n\circ \gamma$ for each $n$. By construction, $\varphi_n(\Db)\subset \Omega$ for each $n$. Recall that $\gamma$ is the restriction of a diffeomorphic embedding of $\Rb$ into $\Db$. As the Poincar{\'e} metric on $\Db$ equals $k_{\Db}$, we have by definition $k_{\Db}(\gamma(t), \gamma^\prime(t)) = 1$ for all $t\in [0, T]$. Thus, by the holomorphicity of $\varphi$, we get for each $n \geq 1$:
\begin{align}
  k_{\Omega}\big(\sigma_n(t), \sigma^\prime_n(t)\big)
  &= k_{\Omega}\big(\varphi_n(\gamma(t)), \varphi^\prime_n(t)\gamma^\prime(t)\big) \notag \\
  & \leq k_{\Db}\big(\gamma(t), \gamma^\prime(t)\big) = 1 \quad \forall t \in [0, T].
  \label{eqn:deriv_1}
\end{align}

Next, observe that for each $n \geq 1$:
$$
  |s-t| - T \leq 0 \leq K_{\Omega}\big(\sigma_n(s), \sigma_n(t)\big)
  \leq K_{\Db}(\gamma(s), \gamma(t)\big) = |s-t|
$$
for every $s, t\in [0,T]$. From the latter inequality and \eqref{eqn:deriv_1}, we conclude that, for each $n$, $\sigma_n$ is a $(1, T)$-almost-geodesic joining $z_n$ to $w_n$. By construction, we can find $\overline{\Omega}^{End}$-open neighborhoods $V_{\xi}$ and $V_{\eta}$ of $\xi$ and $\eta$, respectively, such that
$$
 \clos{V_{\xi}} \cap \clos{V_{\eta}}\,=\,\overline{V}_{\xi} \cap \overline{V}_{\eta} = \emptyset,
$$
and such that $z_n \in V_{\xi}$ and $w_n \in V_{\eta}$ for all sufficiently large $n$. By construction, given a compact $K \subset \Omega$, there exists an integer $N_K \gg 1$ such that $\sigma_n \cap K = \emptyset$ for every $n \geq N_K$. Since each $\sigma_n$ is a $(1, T)$-almost-geodesic, we conclude that $\Omega$ is not a weak visibility domain. But as $(\Omega, K_{\Omega})$ is Gromov hyperbolic and ${\sf id}_{\Omega}$ extends to a homeomorphism from $(\Omega \cup \partial_{G}\Omega)$ onto $\overline{\Omega}^{End}$, the last statement contradicts the conclusion of Theorem~\ref{thm:Gromov_visi}. Therefore, our assumption must be false, whence $\partial\Omega$ does not contain any germs of complex varieties of positive dimension.
\end{proof}

\section{A Wolff-Denjoy theorem: Proof of Theorem~\ref{thm:WD_visible}}
The proof of Theorem~\ref{thm:WD_visible} resembles the proof of Theorem~1.8 in~\cite{BM2021}, which is very similar to the proof of Theorem~1.10 in~\cite{BZ2017}. The present argument requires two preliminary results which are analogous to Proposition 4.1 and Theorem 4.3 in~\cite{BM2021}, but whose proofs are slightly different due to the fact that we must consider unbounded domains as well.
 
Given two complex manifolds $X$ and $Y$, let ${\rm Hol}(X,Y)$ denote the space of holomorphic maps from $X$ to $Y$. A sequence $( f_n)_{n \geq 1}$ in ${\rm Hol}(X,Y)$ is called \emph{compactly divergent} if for every pair of compact subsets $K_1 \subset X$ and $K_2 \subset Y$, the intersection $f_n(K_1) \cap K_2$ is empty for all sufficiently large $n$. 
 
 \begin{lemma}[analogous to Theorem 4.3 in~\cite{BM2021}]\label{lem: analogue of Theorem 4.3} Let $\Omega \subset \Cb^d$ be a weak visibility domain and $X$ be a connected complex manifold. If $( f_n)_{n \geq 1}$ is a compactly divergent sequence in ${\rm Hol}(X,\Omega)$, then there exist $\xi \in \partial \overline{\Omega}^{End}$ and a subsequence $(f_{n_j})_{j \geq 1}$ such that 
 $$
 \lim_{j \rightarrow \infty} f_{n_j}(z) = \xi
 $$
 for all $z \in X$.  
 \end{lemma} 
 
 \begin{proof} Fix $z_0 \in X$. Passing to a subsequence and relabeling, if needed, we can suppose that $\xi_0:=\lim_{n \rightarrow \infty} f_n(z_0)$ exists in $\partial \overline{\Omega}^{End}$. Aiming for a contradiction, suppose that there exists $z_1 \in X$ such that $(f_n(z_1))_{n \geq 1}$ does not converge to $\xi_0$. Passing to a further subsequence, we can suppose that $\xi_1:=\lim_{n \rightarrow \infty} f_n(z_1)$ exists and $\xi_0 \neq \xi_1$. 
 
Fix a smooth curve $\sigma : [0,T] \rightarrow X$ joining $z_0$ to $z_1$. Since the Kobayashi metric is upper-semicontinuous, we can assume that $k_X(\sigma(t); \sigma^\prime(t)) \leq 1$ for every $t \in [0, T]$. Write $\hat{\sigma}_n := f_n \circ \sigma$. Then,
$$
k_{\Omega}(\hat{\sigma}_n(t); \hat{\sigma}_n^\prime(t)) \leq k_X(\sigma(t); \sigma^\prime(t)) \leq 1
$$
for every $t \in [0, T]$ and every $n$. Furthermore:
\begin{align*}
K_{\Omega}(\hat{\sigma}_n(s), \hat{\sigma}_n(t)) &\leq K_X(\sigma(s), \sigma(t)) \leq 
\int_s^t\!k_X(\sigma(\tau); \sigma^\prime(\tau))\,d\tau \leq |s-t|, \\
K_{\Omega}(\hat{\sigma}_n(s), \hat{\sigma}_n(t)) &\geq 0 \geq |s-t| - T
\end{align*}
for every $t \in [0, T]$ and every $n$. Thus, each $\hat{\sigma}_n$ is a $(1,T)$-almost geodesic. By the weak visibility property, there exists a compact set $K \subset \Omega$ such that 
$$
\emptyset \neq K \cap \hat{\sigma}_n([0,T]) = K \cap f_n\big(\sigma([0,T])\big)
$$
for every $n$. Since $(f_n)_{n \geq 1}$ is a compactly divergent sequence, we have a contradiction. 
\end{proof} 
 
\begin{lemma}[analogous to Proposition 4.1 in~\cite{BM2021}]\label{lem: analogue of Prop 4.1} Let $\Omega \subset \Cb^d$ be a weak visibility domain, $o \in \Omega$, and $F : \Omega \rightarrow \Omega$ be a holomorphic self-map. If 
$$
 \limsup_{n \rightarrow \infty} K_\Omega(F^n(o), o) = \infty,
$$
then there exists $\xi \in \partial \overline{\Omega}^{End}$ so that 
$$
 \lim_{j \rightarrow \infty} F^{\mu_j}(o) = \xi
$$
for every sequence  $(\mu_j)_{j \geq 1}$ in $\mathbb{N}$ with $\lim_{j \rightarrow \infty} K_\Omega(F^{\mu_j}(o), o) = \infty$. 
\end{lemma} 
 
The proof is similar to the proof of Proposition 4.1 in~\cite{BM2021}, which is based on an argument of Karlsson~\cite{K2001}. 
 
\begin{proof} Fix an increasing sequence $(\nu_j)_{j \geq 1}$ in $\mathbb{N}$ so that 
\begin{equation}
\label{eqn:defining property of nu} 
 \max\{ K_\Omega(F^{k}(o), o) : k=1,\dots, \nu_j-1\} < K_\Omega(F^{\nu_j}(o), o)
\end{equation}
for every $j \geq 1$. Passing to a subsequence and relabeling, if needed, we can suppose that $\xi : =  \lim_{j \rightarrow \infty} F^{\nu_j}(o)$ exists in $ \partial \overline{\Omega}^{End}$. 
 
Fix a sequence $(\mu_j)_{j \geq 1}$ in $\mathbb{N}$ with $\lim_{j \rightarrow \infty} K_\Omega(F^{\mu_j}(o), o) = \infty$. Aiming for a contradiction, suppose that $(F^{\mu_j}(o))_{j \geq 1}$ does not converge to $\xi$. Passing to a subsequence, we can suppose that $F^{\mu_j}(o) \rightarrow \eta$ and $\eta \neq \xi$.  Next fix a subsequence $(\nu_{i_j})_{j \geq 1}$ such that $\mu_j < \nu_{i_j}$ for all $j \in \Nb$. For each $j \in \Nb$, let $\sigma_j:[0,T_j] \rightarrow \Omega$ be a $(1,1)$-almost-geodesic joining $F^{\mu_j}(o)$ to $F^{\nu_{i_j}}(o)$. By the weak visibility property, there exists a sequence $(t_j)_{j \geq 1}$, $t_j \in [0, T_j]$, such that 
$$
R:= \sup_{j \geq 0} K_\Omega(o, \sigma_j(t_j))
$$
is finite.Then, since each $\sigma_j$ is a $(1,1)$-almost-geodesic, we have 
\begin{align*}
K_\Omega( F^{\nu_{i_j}}(o), F^{\mu_j}(o)) \geq T_j - 1 &\geq 
K_\Omega( F^{\nu_{i_j}}(o), \sigma_j(t_j)) + K_\Omega(\sigma_j(t_j), F^{\mu_j}(o))-3  \\
 & \geq  K_\Omega( F^{\nu_{i_j}}(o), o) + K_\Omega(o, F^{\mu_j}(o)) - 3 - 2R.
 \end{align*}
Further, by the distance decreasing property of the Kobayashi distance under holomorphic maps, and by \eqref{eqn:defining property of nu}, 
 \begin{align*}
 K_\Omega( F^{\nu_{i_j}}(o), F^{\mu_j}(o)) \leq  K_\Omega( F^{\nu_{i_j}-\mu_j}(o), o) < K_\Omega(F^{\nu_{i_j}}(o), o). 
 \end{align*}
 Combining the last two equations we have 
 $$
 K_\Omega(o, F^{\mu_j}(o)) < 3+2R
 $$
 which contradicts the assumption that $\lim_{j \rightarrow \infty} K_\Omega(F^{\mu_j}(o), o) = \infty$.
 \end{proof} 

We are now in the position to give a proof of Theorem~\ref{thm:WD_visible}. Since this proof is nearly identical to the proof of Theorem~1.8 in~\cite{BM2021}, we shall be \textbf{brief.} We should mention here that, although the domains considered in \cite[Theorem~1.10]{BZ2017} and \cite[Theorem~1.8]{BM2021} are visibility domains, the weak visibility property is sufficient for these proofs to work.
 
\begin{proof}[Proof of Theorem~\ref{thm:WD_visible}] The proof of Theorem 1.8 in~\cite{BM2021} applies to the present set-up \emph{essentially} verbatim with Lemma~\ref{lem: analogue of Theorem 4.3} replacing any reference to Theorem 4.3 in~\cite{BM2021}, Lemma~\ref{lem: analogue of Prop 4.1} replacing any reference to Proposition 4.1 in~\cite{BM2021}, and the words ``$\Omega$ is a weak visibility domain'' replacing the phrase ``$\Omega$ is a visibility domain''. The only modifications that occur are the following: 
\begin{itemize}
\item We need a definition of the function $G_\delta$ in the case when $\xi$ is an end of $\overline{\Omega}$. In this case, $G_\delta$ should be defined by 
$$
G_\delta(x_1,x_2) = \inf \left\{ K_\Omega( F^m(x_1), x_2) : m \in \Nb\text{ and } F^m(x_1) \in U_\delta\right\}
$$
where $U_\delta$ is the connected component of $\overline{\Omega} \setminus \overline{\Bb_d(0,\delta^{-1})}$ that contains the end $\xi$. 
\item Given a compact subset $K \subset \Omega$ and a point $\eta \in \partial \overline{\Omega}^{End}$, we need to know that the quantity 
$$
\epsilon:=\liminf_{z \rightarrow \eta} \inf_{k \in K} K_\Omega(k,z) 
$$
is contained in $(0,\infty]$. In the unbounded case this requires a slightly different argument than the one given in~\cite{BZ2017} or ~\cite{BM2021}. In our setting we can argue as follows: fix a compact set $K^\prime \subset \Omega$ such that $K \varsubsetneq {\rm int}(K^\prime)$. As $\Omega$ is Kobayashi hyperbolic, the Euclidean topology and the $K_{\Omega}$-topology coincide (see \cite[Section~3.3]{JP1993}, for instance). So, $\Omega \setminus {\rm int}(K^\prime)$ is closed in the $K_{\Omega}$-topology and $K\ni k \mapsto K_{\Omega}(k, \Omega \setminus {\rm int}(K^\prime))$ is a continuous function that is positive at each $k \in K$. As $K$ is compact, 
$$
\epsilon^\prime:= \inf_{k \in K} K_{\Omega}(k, \Omega \setminus {\rm int}(K^\prime)) 
$$
is positive. Then $\epsilon \geq \epsilon^\prime > 0$ and the argument is complete.
\end{itemize}
\end{proof}

\section*{Acknowledgements}    
Bharali is supported by a UGC CAS-II grant (Grant No. F.510/25/CAS-II/2018(SAP-I)). Zimmer was partially supported by grants DMS-2105580 and DMS-2104381 from the National Science Foundation. We are very grateful for the helpful suggestions of the referees for enhancing the clarity of some of our proofs.

\bibliographystyle{alpha}
\bibliography{complex_kob}

\end{document}